\documentclass[fleqn,11pt]{article}

\usepackage{latexsym, amsmath, amssymb, euscript}
\usepackage{fontenc, indentfirst}
\usepackage{fancyhdr}
\usepackage{amsmath,amssymb}
\usepackage{multicol}

\usepackage{theorem}
\usepackage{amssymb, amsmath, amsbsy, amsfonts}
\usepackage{stmaryrd}

\usepackage{graphics,graphicx}
\usepackage[usenames,dvipsnames]{xcolor}

\newtheorem{theorem}{Theorem}[section]
\newtheorem{lemma}[theorem]{Lemma}
\newtheorem{proposition}[theorem]{Proposition}
\newtheorem{definition}[theorem]{Definition}
\newtheorem{corollary}[theorem]{Corollary}
{\theorembodyfont{\rmfamily}
\newtheorem{remark}[theorem]{Remark}
}

\newcommand{\N}{\mathbb{N}}
\newcommand{\Z}{\mathbb{Z}}
\renewcommand{\S}{\mathbb{S}}
\newcommand{\R}{\mathbb{R}}
\newcommand{\T}{\mathbb{T}}

\newcommand{\E}{\mathcal{E}}
\newcommand{\Ext}{{\rm Ext}}

\newcommand{\Om}{\Omega}

\newcommand{\Tr}{{\rm Tr}}

\definecolor{ao(english)}{rgb}{0.0, 0.5, 0.0}

\newcommand{\ol}[1]{\overline{#1}}

\definecolor{ao(english)}{rgb}{0.0, 0.5, 0.0}

\newcommand{\footnoteremember}[2]{\footnote{#2}\newcounter{#1}\setcounter{#1}{\value{footnote}}}

\def\qed{\hfill $\square$ \goodbreak \medskip}
\def\eps{\varepsilon}

\title{Estimates of First and Second Order Shape Derivatives\\ in  Nonsmooth Multidimensional Domains and Applications}

\author{Jimmy Lamboley\footnoteremember{1}{CEREMADE, Universit\'e Paris-Dauphine, Place du Mar\'echal de Lattre de Tassigny, 75775 Paris, France}, Arian Novruzi\footnote{University of Ottawa, Department of Mathematics and Statistics, 585 King Edward, Ottawa, ON, K1N 6N5, Canada}, Michel Pierre\footnote{ENS Rennes, IRMAR, UEB, av Robert Schuman, 35170 Bruz, France}}
\date{\today}

\begin{document}
\maketitle

\begin{abstract}
In this paper we investigate continuity properties of first and second order shape derivatives of functionals depending on second order elliptic PDE's around nonsmooth domains, essentially either Lipschitz or convex, or satisfying a uniform exterior ball condition. We prove rather sharp continuity results for these shape derivatives with respect to Sobolev norms of the boundary-traces of the displacements. With respect to previous results of this kind, the approach is quite different and is valid in any dimension $N\geq 2$. It is based on sharp regularity results for Poisson-type equations in such nonsmooth domains. We also enlarge the class of functionals and PDEs for which these estimates apply. Applications are given to qualitative properties of shape optimization problems under convexity constraints for the variable domains or their complement. \\

\smallskip

\end{abstract}

\section{Introduction}\label{sect:intro}

In this paper, we focus on regularity estimates for first and second order shape derivatives {\em around nonsmooth subsets of $\R^N$ {($N\geq 2$)}} for energy functionals involving classical elliptic {partial differential equations (PDE)}. For instance, we address the following question: 
given a bounded {Lipschitz or convex} subset $\Omega_0\subset\R^N$, {we wonder} what is the {optimal} regularity of the shape derivatives 
$$\xi\to E'(\Omega_0)(\xi),\;\; \xi\to  E''(\Omega_0)(\xi,\xi),$$
where $E'(\Omega_0), E''(\Omega_0)$ respectively denote the first  and second shape derivatives around $\Omega_0$ of the shape functional {$\Omega\mapsto E(\Omega)=\int_\Omega K(x,U_\Omega,\nabla U_\Omega)dx$, $K(x,\cdot,\cdot)$ a quadratic polynomial and $U_\Omega=U_\Omega(x)$} the solution of an elliptic PDE set in $\Omega$ (see Sections \ref{ssect:energy}, \ref{ss:mainresults} for the precise definitions). 

This question, interesting for itself, is in particular motivated by the qualitative analysis of shape optimization problems of the form
\begin{equation}\label{e:minJ}
\min\{J(\Om),\ \Om\subset\mathbb R^N\ \textrm{ \bf convex},\ \Omega\in \mathcal{O}_{ad}\},
\end{equation}
where $\mathcal{O}_{ad}$ is a set of admissible subsets of $\R^N$ and $J:\mathcal{O}_{ad}\rightarrow\R$ is a shape functional which itself involves shape functionals $\Omega\mapsto E(\Omega)$ of the above type.

The following 2-dimensional example was considered in \cite{LNP} among other cases:
\begin{equation}\label{e:J,Omad}
J(\Om)={R}(E(\Omega),|\Omega|)-P(\Omega),\quad
\mathcal{O}_{ad}=\left\{\Omega\subset\mathbb R^2,\;\textrm{ open and }\,\partial\Omega\subset \{x,\; a\leq|x|\leq b\}\right\},
\end{equation}
where ${R}:\R^2\to\R$ is a smooth function, {$E(\Omega)$ is a shape functional related to a PDE}, $|\Omega|$ is the Lebesgue measure of $\Omega$,
$P(\Omega)$ its perimeter and $(a,b)\in[0,\infty]^2$.
It was proved (see \cite[Theorem 2.9, Theorem 2.12 and Corollary 2.13]{LNP} and also Section \ref{sect:app} in the present paper) that if the second order shape derivative $\xi\mapsto E''(\Om_0)(\xi,\xi)$ around any bounded convex subset $\Om_0$ is {\em continuous with respect to a  norm strictly weaker than $H^1(\partial\Om_0)$}, then optimal shapes of \eqref{e:minJ} are {\em polygonal} in $\left\{x,\; a<|x|< b\right\}$.
In \cite{LNP}, the authors prove that such a continuity holds in the two specific examples:
when the functional {$E(\Omega)$ is the Dirichlet energy of $\Omega$ -that is $K=K(x,U,q)=\|q\|^2-2f(x)$ and $U_\Omega$ is solution of the associated Dirichlet problem}, or when $E(\Omega)$ is the first eigenvalue of the Dirichlet-Laplacian on $\Omega$. More precisely, it is proved in \cite{LNP}  that the second order shape derivative of $E(\cdot)$ is, in these two examples and around any open convex domain $\Om_0$, continuous for the $H^{1/2}(\partial\Om_0)\cap L^\infty(\partial\Om_0)$ topology (and therefore continuous for the $H^{1/2+\eps}(\partial\Om_0)$-topology for any $\eps>0$). The proof of this continuity strongly relies on the 2-dimensional environment. 
{As explained below, we prove in this paper that {\em even the full $H^{1/2}(\partial\Omega_0)$-continuity holds in this case and even in any dimension} (see iii) in Corollary \ref{c:e',e''(i,e,intro)}).}
{Note that this continuity is optimal since, for regular convex domains $\Omega_0$, if for instance $f=0$ on $\partial\Omega_0$, then $E''(\Omega_0)$ satisfies 
\begin{equation}\label{coerc}
E''(\Omega_0)(\xi,\xi)\geq C\|\xi\|^2_{H^{1/2}(\partial\Omega_0)},
\end{equation}
for any displacement $\xi$ which is orthogonal to $\partial\Omega_0$. This estimate (\ref{coerc}) may be obtained by using (5.101) and Section 5.9.6. in \cite{HP}. 

In Problem (\ref{e:minJ}), only convex domains are involved. However, it is also interesting to consider shape optimization problems where the PDE is set in an {\em "exterior domain"} like
\begin{equation}\label{ext1}
\min\{J(\R^N\setminus\overline{\Om}),\ \Om\subset\mathbb R^N\ \textrm{ \bf convex},\ \Omega\in \mathcal{O}_{ad}\},
\end{equation}
where $\mathcal{O}_{ad}$ is as before, but now $J$ involves $E(\Omega)=\int_{\R^N\setminus\overline{\Omega}}K(U_{\Omega},\nabla U_{\Omega})$ where $U_{\Omega}$ is solution of an elliptic PDE set on the {\em exterior domain} $\R^N\setminus\overline{\Omega}$, which is here the complement of a convex set. It is well-known that solutions of such PDE's are not so regular as in convex domains. These exterior domains are nevertheless Lipschitz domains and this is one main reason why it is interesting to look at the case of Lipschitz domains even for shape optimization problems involving convex domains like (\ref{ext1}).\\

In this paper,  we use a different strategy to estimate shape derivatives, and we generalize and improve the above-mentioned shape derivative estimates in the following directions.

\begin{itemize}
\item A most important generalization concerns the dimension of shapes. In \cite{LNP}, the result and the strategy were restricted to {\bf planar shapes}. Thanks to our new strategy, we are able to provide estimates of first and second order shape derivatives {\bf in any dimension}. This is an important breakthrough for the study of problems like \eqref{e:minJ} in dimension 3 or higher.

\item We generalize the class of PDE energy-functionals in several ways (see Section \ref{ssect:energy} for precise definitions): the underlying elliptic operator is now a general linear elliptic operator with variable coefficients, the energy is any quadratic integral functional of the state function and its gradient (and in particular does not need to be the energy associated to the PDE defining the state function) and more importantly, we consider both interior and exterior PDEs. This last point motivates the next item.

\item We generalize the classes of shapes we consider. Indeed we do not only consider the class of convex domains, but we investigate two wider classes, namely the class of Lipschitz domains and the class of Lipschitz domains satisfying a uniform exterior ball condition (we refer to these domains as {\it {semi}-convex} domains, see Definition \ref{def1}). Even if we are interested in applications about convex domains, as we pointed out in the previous point, we are interested in PDEs defined in the exterior of some convex domain. Whence the consideration of the above kinds of boundary regularity. Estimates for semi-convex domains are the same as those for convex domains. They are weaker for Lipschitz domains, but {probably} rather sharp and they are strong enough to be used for our applications (see Section \ref{sect:app}) which are interesting for themselves.

\item Even in the particular case where $E$ is exactly the Dirichlet energy and $\Om$ is a 2-dimensional convex domain, we improve the result of \cite{LNP} in obtaining that $E''(\Om)$ is continuous in the $H^{1/2}(\partial\Om)$-norm, instead of $H^{1/2}(\partial\Om)\cap L^\infty(\partial\Om)$. As noticed above, this result is sharp since one cannot expect continuity in an $H^s$-norm for $s<1/2$ (even if $\Om$ is smooth). This specific result is actually valid as soon as the energy is the one associated to the PDE defining the state function, and is valid in any dimension and for any semi-convex domain.

\end{itemize}

The method for proving the shape derivative estimates is new. In order to study the derivative at a set $\Om_{0}$, 
we first prove adequate estimates of the ``material derivative'' $\hat{U}_\theta'$, where $\theta:\R^N\to\R^N$ is a smooth vector field, $'$ denotes 
the derivative with respect to $\theta$,
$\hat{U}_\theta=U_\theta\circ(I+\theta)$, and $U_\theta$ is the state solution of the related PDE in the domain $\Omega_\theta=(I+\theta)(\Omega_0)$.
For our purpose, it appears that it is much more efficient than using the usual shape derivative $U'_\theta$. We use the well-known property 
(see for example \cite{HP}) that without any regularity of $\Om$,
the map $\theta\mapsto\hat{U}_\theta$ is in general differentiable (even $C^\infty$ if the involved coefficients are smooth) when seen as a 
$H^1_0(\Omega_0)$-valued function
(whereas $\theta\mapsto U'_\theta$ is differentiable only when seen as an $L^2$-valued function). We show how the regularity of the shape derivatives 
essentially depends on the regularity of the state function $U_0$. Thus we shall use some sharp regularity results for the solution of a linear second 
order PDE in a Lipschitz or semi-convex domain, in particular $W^{1,p}$-regularity results when the data are in $W^{-1,p}$ (see Propositions 
\ref{p:LU=f(i),Lip} and \ref{p:LU=f(i),conv}).

Note that we insisted in this introduction on the second order shape derivative, but we also obtain estimates of the first order shape derivative 
which seem to be new as well, in this non-smooth setting.

As an application of these estimates, using the strategy of \cite{LN,LNP}, we prove that any solution $\Omega_0$ of (\ref{e:minJ}), (\ref{ext1}) is 
polygonal, if $N=2$ and if $E$ is in one of two classes described above (where $E$ depends on the solution of a PDE in the interior or the exterior of a 
convex domain, see more precisely Section \ref{ssect:energy}).
Also, in higher dimension, we use our estimates to analyze solutions of the $N$-dimensional version of Problem \eqref{e:J,Omad}. We obtain very strong qualitative properties of optimal shapes, namely that the space of deformations which leave it convex is actually {\em of finite dimension} (see Theorem \ref{th:multiD}). As an easy consequence, we obtain that any optimal shape has zero Gauss curvature on any open set where the boundary is smooth. Actually, this "finite dimension" property does contain quite more geometrical information.

This paper is structured as follows. In Section \ref{sect:results} we introduce the notations, the classes of PDE's under consideration and 
state our main results. In Section \ref{s:estimates} we prove the estimates of first and second order shape derivatives. In the last section, we apply these estimates for the analysis of optimal planar convex domains, which happen to be polygons and we conclude by analyzing the consequence of our estimates for higher dimensional optimal convex domains.

\section{Main results}\label{sect:results}

\subsection{Class of energy functionals}\label{ssect:energy}

We deal with energy functionals $E(\Omega)$ of two main forms.
\begin{itemize}
\item
\noindent{\bf Interior PDE:}
$E(\Omega)$ depends on a PDE in the interior of $\Omega$, namely
\begin{equation}\label{e:E(i)}
E(\Omega)=
\int_\Omega 
 K(x,U_\Omega,\nabla U_\Omega)dx,
\end{equation}
where
\begin{equation}
K(x,U,q)
=
\frac{1}{2}
\left(
\sum_{i,j=1}^N
\alpha_{ij}(x)q_{i}q_{j}
+
\alpha_{00}(x)U^2
\right)
+
\sum_{i=1}^N\beta_{i}(x)q_{i}
+
\gamma(x) U
+
\delta(x),										\label{e:K(i)}
\end{equation}
with smooth enough coefficients $\alpha,\beta,\gamma,\delta$, and $U_\Omega$ solution of 
\begin{equation}
U\in H^1_0(\Omega),\;\;  LU:
=
 -\sum_{i,j=1}^N\partial_j(a_{ij}(x)\partial_iU) + \sum_{i=1}^Nb_i\partial_i U + cU
 =
 f,\;\;\;\;\;\;\textrm{ in }\Om.								\label{e:LU=f(i)}
\end{equation}

Regularity of the coefficient will be made precise later, but they will satisfy the following throughout the paper: 
\begin{eqnarray}
&&
\exists\lambda>0,\quad
\forall\xi\in\mathbb R^N,\; \forall x\in\mathbb R^N,\quad \lambda|\xi|^2
\leq
\sum_{i,j=1}^N
a_{ij}(x)\xi_i \xi_j
\leq
\lambda^{-1}|\xi|^2,\quad
		\label{e:a(i)}\\
&&
\forall (i,j)\in\llbracket1,n\rrbracket,\;\; a_{ij}=a_{ji},\;\;
\alpha_{ij}=\alpha_{ji}, \;\;\;\;\textrm{ and }c\geq 0.				\label{e:Mij=Mji}			
\end{eqnarray}
Note that the condition $\alpha_{ij}=\alpha_{ji}$ is not restrictive as, in the case the matrix $(\alpha_{ij})$ is not symmetric, we can consider
$\overline{\alpha}_{ij}=\frac{1}{2}(\alpha_{ij}+\alpha_{ji})$, which is symmetric, and then in (\ref{e:K(i)}) take $\overline{\alpha}_{ij}$ instead of $\alpha_{ij}$. {Note also that no ellipticity condition is a priori required on $(\alpha_{ij})$.} Actually, our strategy can handle much more general functional $K$, see Remark \ref{r:|e'|+|e''|(i),Om-conv,K-gen(i)}.

\item
\noindent{\bf Exterior PDE:} {For simplicity, we will restrict ourselves to the model example where}
\begin{equation}\label{e:E(e)}
E(\Omega)
=
\int_{\Omega^e} |\nabla U_\Omega|^2,
\end{equation}
where $U_\Omega\in H^{1,0}_0(\Omega^e)$ solves
\begin{equation}\label{e:LU=f(e)}
-\Delta U=f\textrm{ in }\Om^e:=\mathbb R^N\backslash\overline{\Omega}.
\end{equation}
and $H^{1,0}_0(\Omega^e)$ is defined (see \cite{nedelec-1}) as the set of functions of $H^{1,0}(\Omega^e)$ with trace 0 on $\partial\Om$, where
\begin{eqnarray}
\hspace*{-9mm}
 H^{1,0}(\Omega^e)
 \!&\!=\!&\!
 {\bigg\{} 
  U:\Omega^e\mapsto\mathbb R,
  						\label{e:H10}\\
  \!&\!\!&\!
  \left.
  \hspace*{1mm}
 \|U\|^2_{H^{1,0}(\Omega^e)}
  :=
  \left\|\frac{U}{(1+|x|^2)^{1/2}(\ln(2+|x|^2))^{\delta_{2,N}/2}}\right\|^2_{L^2(\Omega^e)}	
  \!\!\!\!\!+ 
  \|\nabla U\|^2_{L^2(\Omega^e)} <\infty
  \right\},									\nonumber
\end{eqnarray}
where $\delta_{2,N}=1$ if $N=2$ and $\delta_{2,N}=0$ if $N\neq 2$.
\end{itemize}

\subsection{Shape derivative estimates}\label{ss:mainresults}

In the whole paper, we consider 
{\begin{equation}\label{Omega}
\Om_0, \Om\;\Subset B_R=B(0,R), \;  {\rm connected,\; open,\;  bounded},
\end{equation}
  }
 where $R>0$ (and large). We set $\Theta:=W^{1,\infty}_0(B_R,\R^N)$ the Banach space {equipped with the usual norm
$$\forall\theta\in \Theta,\; \|\theta\|_{1,\infty}:=\sup_{x\in B_R}\{\|\theta(x)\|+\|D\theta(x)\|\}.$$}
 Given a  shape functional $E(\cdot) :\mathcal{O}_{ad}\to \R$ defined on a family $\mathcal{O}_{ad}$ of admissible subsets of $\R^N$, we consider
$$\mathcal{E}:\Theta\to\R, \;\;\E(\theta):=E((I+\theta)(\Om_{0})).$$
Then, {$E$ is said to be shape differentiable of order $m\in\N^*$ at $\Omega_0$ (resp. of class $C^m$ near $\Omega_0$) if and only if $\E$ is $m$ times Fr\'echet-differentiable at $\theta=0$ (resp. if and only if $\E$ is $m$ times continuously differentiable in a neighborhood of $\theta=0$)}. When it is well-defined, $\E'(0)$ and $\E''(0)$ are respectively called the first and second order shape derivative of $E$ at $\Om_{0}$, and can also be denoted by $E'(\Om_{0})$ and $E''(\Om_{0})$. Their values at {$\xi,\eta\in\Theta$ are denoted by $\E'(0)(\xi)$ or $E'(\Omega)(\xi)$, $\E''(0)(\xi,\eta)$ or $E''(\Omega)(\xi,\eta)$}. They are linear and bilinear continuous forms on $\Theta$ respectively
and it is well-known that $E'(\Om_{0}), E''(\Om_{0})$ only depend on the trace on $\partial\Om_{0}$ of the deformations {$\xi$ and $\eta$} (see the proof of Theorem \ref{th:e',e''(i,intro)}).\\

\begin{definition}\label{def1}
{A bounded open subset  $\Om\subset \R^N$ is said to be semi-convex if it is Lipschitz and satisfies a {\em uniform exterior ball condition} in the following sense: there exists $r>0$ such that for any $x\in\partial\Om$, there exists $y\in\R^N$ with $\overline{B(y,r)}\cap\overline{\Om}=\{x\}$.}
\end{definition}
{It is known that a domain is semi-convex if it is locally representable as the graph of a semi-convex function (see for example \cite[Theorem 3.9]{MMY}), where a function $f:C\to\R$ is said semi-convex on a convex subset $C$ of $\R^n$ if there exists $M\in\R$ such that $x\in C\mapsto f(x)+M\|x\|^2$ is convex.}
Recall that $\Omega$ is said to be Lipschitz if it is locally representable as the graph of a Lipschitz function.
\begin{remark}\label{r:dR->qc} 

{It is easy to check that, if $\Omega$ is semi-convex then for $R_1$ small enough, there exist $(\delta_1(R_1), \sigma(R_1))\in (0,\infty)\times (0,1)$ with $\lim_{R_1\to 0} \delta_1(R_1)=0$ such that $\Omega$ is a 
$(\delta_1(R_1), \sigma(R_1), R_1)$-quasi-convex domain in the sense of \cite{JLW}.
It follows that $\Omega$ is a $(\delta, \sigma, R)$-quasi-convex domain in the sense of \cite{JLW} for all 
$(\delta,\sigma,R)\in [\delta_1(R_1),\infty)\times(0,k(R)\sigma(R_1)]\times (0,R_1]$, { with $k(R)>0$, $k(R_1)=1$, $\lim_{R\to0}k(R)=0$}.

On the other hand, given a Lipschitz matrix $(a_{ij})$ (i.e. $a_{ij}\in W^{1,\infty}(\Omega)\; \forall i,j=1,...,N$), then for all $R_2\in (0,1)$, there exists $\delta_2(R_2)\in (0,\infty)$ with 
$\lim_{R_2\to 0} \delta_2(R_2)=0$ such that $A$ is a $(\delta_2(R_2), R_2)$-vanishing matrix in the sense of \cite{JLW}. It follows that $A$ is $(\delta,R)$-vanishing matrix in the sense of \cite{JLW} for all $(\delta,R)\in[\delta_2(R_2),\infty)\times (0,R_2]$. 

Therefore, we will be able to apply the $W^{1,p}$-regularity results proved in Theorem 1.1 of \cite{JLW} for the solutions to (\ref{e:LU=f(i)}) in quasi-convex domains (see Proposition \ref{p:LU=f(i),conv}). 
Indeed, given $p\in (1,\infty)$, we first choose $R_1$ and $R_2$ small enough so that $\delta_1(R_1),\delta_2(R_2)\leq \delta(N,p)$ as defined in Theorem 1.1 of \cite{JLW}, and $\Omega$ is a $(\delta_1(R_1),\sigma(R_1),R_1)$-quasi-convex and $(a_{ij})$ is 
{ $(\delta(R_2),R_2)$-vanishing}. It follows that if $R=\min\{R_1,R_2\}$ and { $\sigma=k(R)\sigma(R_1)$} then
$\Omega$ is a $(\delta(N,p),\sigma,R)$-quasi-convex and $(a_{ij})$ is 
$(\delta(N,p),R)$-vanishing.
}
\end{remark}

{Let us now present the main  results of this paper.}
\begin{theorem}\label{th:e',e''(i,intro)}
{\bf (interior)}
{Let $\Omega_0$ be as in (\ref{Omega}).} For $\theta\in\Theta$, 
we denote by $U_\theta$ the solution of (\ref{e:LU=f(i)}) in $\Omega_\theta=(I+\theta)(\Om_{0})$ (see  Proposition \ref{p:LU=f(i),Lip})  
and $\E(\theta)=E(\Omega_\theta)$, where $E$ is given by (\ref{e:E(i)}). The following holds.
\begin{itemize}
\item[i)]
If $a_{ij}, b_i, c, f,\ \alpha_{ij}, \beta_i, \gamma, \delta \in W^{m,\infty}({B_{R}})$, $m\in\N^*$, 
then $\left[\theta\mapsto\E(\theta)\right]$ is of class $C^m$ near $\theta=0\in\Theta$.
\end{itemize}
{\it In the rest of this statement, we assume the previous hypotheses are satisfied {for $m=2$. }\begin{itemize}{ \item[ii)] \label{item:L}
Assume $\Omega_0$ is Lipschitz and $L$ is of the form \eqref{e:LU=f(i)}. Then there exists ${\ol r}_1={\ol r}_1(\Omega_{0})$ satisfying 
${\ol r}_{1}\in(1,2)$ if $N=2$ and ${\ol r}_{1}\in(1,3)$ if $N\geq 3$ such that, for all $r\in({\ol r}_1,\infty)$, there exist $C_i=C_i(\Omega_0,L,f,K,r)$ with
\begin{eqnarray}
 |E'(\Om_{0})(\xi)|
 &\leq& 
 C_1
\|\xi\|_{W^{1-1/r,r}(\partial\Omega_0)},					\label{e:|e'(i)|<Tr,Lip}
\\
 |E''(\Om_{0})(\xi,\xi)|
 &\leq& 
 C_2
\|\xi\|^2_{W^{1-1/(2r),2r}(\partial\Omega_0)},				\label{e:|e''(i)|<Tr,Lip}
\end{eqnarray}
for all $\xi\in\Theta$.
}
\item[iii)]
Assume $\Omega_0$ is {semi}-convex and $L$ is of the form \eqref{e:LU=f(i)}. Then (\ref{e:|e'(i)|<Tr,Lip}), (\ref{e:|e''(i)|<Tr,Lip}) hold for { all} $r\in(1,\infty)$.
\item[iv)] Assume $\Omega_0$ is {semi}-convex, $L$ of the form \eqref{e:LU=f(i)} with $L^*=L$ and $E$ is exactly its associated energy (meaning that $\partial_{U}K(\cdot,U,\nabla U)-\nabla\cdot\partial_{q}K(\cdot,U,\nabla U)=L U$, 
for every smooth $U$). Then (\ref{e:|e'(i)|<Tr,Lip}), (\ref{e:|e''(i)|<Tr,Lip}) hold for { all} $r\in[1,\infty)$.
\end{itemize}}
\end{theorem}

\begin{proof} This result is a consequence of Theorem \ref{th:|e'|+|e''|(i)} which provides an estimate in terms of ``interior'' norms of $\xi$, 
{ which after replacing }$p\in(2,\ol{p}_{1})$ by $r=p/(p-2)\in({\ol r}_{1}:=\ol{p}_{1}/(\ol{p}_{1}-2),\infty)$ gives
\begin{equation}\label{intestim}
 |E'(\Om_{0})(\xi)|
 \leq 
 C_1
\|\xi\|_{W^{1,r}(\Omega_0)},
\;\;
 |E''(\Om_{0})(\xi,\xi)|
 \leq 
 C_2
\|\xi\|^2_{W^{1,2r}(\Omega_0)}.
\end{equation}
{If $\Om_0$ is { semi}-convex then we can have any $p\in(2,\infty)$ which leads to any $r\in(1,\infty)$, and if moreover $L$ is self-adjoint and $K$ is the associated energy, then we can have $p=\infty, r=1$.
}

{Now, we use the well-known fact, that $E'(\Omega_0)\xi, E''(\Omega_0)(\xi,\xi)$ depend only on the values of $\xi$ on $\partial\Omega_0$. More precisely, 
$$\xi,\hat{\xi}\in W^{1,\infty}(B_R),\; \xi-\hat{\xi}\in W_0^{1,2r}(\Omega_0)\;\Rightarrow\; E'(\Omega_0)(\xi)=E'(\Omega_0)(\hat{\xi}),\;E''(\Omega_0)(\xi,\xi)=E''(\Omega_0)(\hat{\xi},\hat{\xi}).$$
Indeed, let $\zeta_n$ be a sequence in $C_0^\infty(\Omega_0,\R^N)$ converging to $\xi-\hat{\xi}$ in $W^{1,2r}(\Omega_0)$. Let $z_n\in C^1([0,\infty):\R^N)$ be the solution of
$$\forall t\geq 0,\; \forall x\in \R^N,\; \partial_tz_n(t,x)=\zeta_n(z_n(t,x)),\;\; z_n(0,x)=x.$$
{Note that $z_n(t,\cdot)\in \Theta$ (see \cite{L}), and}  for $x$ in some neighborhood of $\partial\Omega_0$, $z_n(t,x)=x$ for all $t\geq 0$. It follows that $z_n(t,\Omega_0)=\Omega_0$, {where 
$z_n(t,\Omega_0)=\{z_n(t,x),\, x\in\Omega_0\}$}, and therefore $E(z_n(t,\Omega_0))=E(\Omega_0)$ for all $t\geq 0$. In particular
\begin{eqnarray*}
E'(\Omega_0)(\xi-\hat{\xi})
&=&
{\lim_{n\to\infty}}\frac{d}{dt}_{|_{t=0}}E(z_n(t,\Omega_0))=0,\\
E''(\Omega_0)(\xi-\hat{\xi},\xi-\hat{\xi})
&=&
{\lim_{n\to\infty}}\frac{d^2}{dt^2}_{|_{t=0}}E(z_n(t,\Omega_0))=0.
\end{eqnarray*}
From this property and (\ref{intestim}), we deduce
$$|E'(\Omega_0)(\xi)|\leq C_1\inf\{\|\hat{\xi}\|_{W^{1,r}(\Omega_0)}: \hat{\xi}\in W^{1,\infty}(\Omega_0;\R^N),\; \hat{\xi}=\xi\;on\;\partial\Omega_0\},$$
and the similar property for $E''(\Omega_0)(\xi,\xi)$ with the $W^{1,2r}$-norm. We then apply Lemma \ref{lem:extension} to deduce \eqref{e:|e'(i)|<Tr,Lip}, \eqref{e:|e''(i)|<Tr,Lip}.
}
\end{proof}

\begin{remark}
The result i) of Theorem  \ref{th:e',e''(i,intro)} a priori implies that $E'(\Om_{0})$ and $E''(\Om_{0})$ are continuous  on $\Theta$ for its $W^{1,\infty}(\Omega_0)$-norm (as a linear and a bilinear form respectively). {Using the previous trace analysis, we deduce that they are also continuous for
the $W^{1,\infty}(\partial\Omega_0)$-norm.}
The estimates  (\ref{e:|e'(i)|<Tr,Lip}), (\ref{e:|e''(i)|<Tr,Lip}) improve this a priori information and show that $E'(\Om_{0}), E''(\Om_{0})$ are continuous in
$W^{1-1/r,r}(\partial\Omega_0)$ and $W^{1-1/(2r),2r}(\partial\Omega_0)$ respectively: note that the regularity gets stronger as $r$ gets smaller.

{In case iv),  the estimate (\ref{e:|e''(i)|<Tr,Lip}) is valid with $r=1$ (which yields the $H^{1/2}$-continuity) and cannot be improved since, as already noticed }{ in the introduction (see (\ref{coerc})), $E''(\Omega_0)$ has $H^{1/2}$-coercivity properties.}

Again in case iv), it is interesting to notice that \eqref{e:|e'(i)|<Tr,Lip} with $r=1$ can be stated as $E'(\Om_{0})\in (L^1(\partial\Om_{0}))'=L^\infty(\partial\Om_{0})$. Again this is sharp in general; if for example $E$ is the Dirichlet energy,
$$E(\Omega)= \frac{1}{2}\int_{\Omega}|\nabla U_\Omega|^2-\int_\Omega f\,U_\Omega,$$
it is well-known that $E'(\Om_{0})=-\frac{1}{2}|\nabla U_{\Om_{0}}|_{|\partial\Om_{0}}^2$, which indeed belongs to $L^\infty(\partial\Om_{0})$.
\end{remark}

\begin{theorem}\label{th:e',e''(e,intro)}
{\bf (exterior)}
Let $\Omega_0$ be as in (\ref{Omega}) and Lipschitz and let $f\in W^{m,\infty}(\mathbb R^N)$, $m\in\N^*$, with ${\rm supp}(f)\Subset B_R$. For $\theta\in\Theta$, let $U_\theta\in H^{1,0}_0(\Omega^e_\theta)$ be the solution of (\ref{e:LU=f(e)}) in $\Omega_\theta^e$ (see Proposition \ref{p:LU=f(e)+})
and let $\E(\theta)= E(\Omega_\theta)$ with $E$ given by  (\ref{e:E(e)}).
Then $\E$ is of class $C^m$ near $\theta=0\in\Theta$. Moreover, there exists ${\ol r}_1={\ol r}_1(\Om_{0})$ satisfying 
{${\ol r}_{1}\in(1,2)$ if $N=2$ and ${\ol r}_{1}\in(1,3)$} if $N\geq 3$ such that, for any 
$r\in \left({\ol r}_1,\infty\right)$, there exist  $C_i=C_i(\Omega_0,U_0,R,f,r)$ with
\begin{eqnarray}
 |E'(\Om_{0})(\xi)|
 &\leq& 
 C_1
\|\xi\|_{W^{1-1/r,r}(\partial\Omega_0)}\;\;\mbox{\it if}\;\; m\geq 1,				\label{e:|e'(e)|<Tr,Lip}
\\
 |E''(\Om_{0})(\xi,\xi)|
 &\leq& 
 C_2
\|\xi\|^2_{W^{1-1/(2r),2r}(\partial\Omega_0)}\;\;\mbox{\it if}\;\; m\geq 2,				\label{e:|e''(e)|<Tr,Lip}
\end{eqnarray}
for any $\xi\in \Theta$.
\end{theorem}

\begin{proof}
Similarly to the proof of Theorem \ref{th:e',e''(i,intro)}, this result is a consequence of Theorem \ref{th:e',e''(e)}, combined with the trace result of Lemma \ref{lem:extension}.
\end{proof}

Previous estimates 
 can be written using a different range of Sobolev spaces. In the following statement, we write these estimates in the $H^s$-norms, which are relevant for our applications (see Section \ref{sect:app}).

 \begin{corollary}\label{c:e',e''(i,e,intro)}
{Let $\Omega_0$ be as in (\ref{Omega}).}
\begin{itemize}
\item[i)]
Under the same conditions as in ii) of Theorem  \ref{th:e',e''(i,intro)},   or as in Theorem \ref{th:e',e''(e,intro)}, {and with $\overline{r}_1=\overline{r}_1(\Omega_0)$ as in these theorems, we have}
\begin{equation}
|E'(\Om_{0})(\xi)|
 \leq 
 C_1
\|\xi\|_{H^{s}(\partial\Omega_0)},\;{\forall s\in(s_{1},1],\; {s_{1}=\min\left\{1\,,\,1-\frac{1}{{\ol r}_{1}}+(N-1)\left(\frac{1}{2}-\frac{1}{{\ol r}_{1}}\right)^+\right\},}}\label{e:|e'(i,e)|<Tr,Lip++}
\end{equation}
\begin{equation}
|E''(\Om_{0})(\xi,\xi)|
 \leq 
 C_2
\|\xi\|^2_{H^{s}(\partial\Omega_0)}, \;{ \forall s\in(s_{2},1],\;		\label{e:|e''(i,e)|<Tr,Lip+}
{s_{2}=\min\left\{1\,,\, \frac{N+1}{2}-\frac{N}{2{\ol r}_{1}}\right\},}}
\end{equation}
for every $\xi\in\Theta$. 
\item[ii)]
{ Under the same conditions as in iii) of Theorem  \ref{th:e',e''(i,intro)}},  (\ref{e:|e'(i,e)|<Tr,Lip++}) (resp. (\ref{e:|e''(i,e)|<Tr,Lip+})) holds for
every $s\in(0,1]$ (resp. $s\in(1/2,1]$).
\item[iii)] 
{ Under the same conditions as in iv) of Theorem  \ref{th:e',e''(i,intro)}},  (\ref{e:|e'(i,e)|<Tr,Lip++}) (resp.  (\ref{e:|e''(i,e)|<Tr,Lip+})) holds also for  $s=0$ (resp. $s=1/2$).
\end{itemize}
\end{corollary}

\begin{proof}
Inequalities (\ref{e:|e'(i,e)|<Tr,Lip++}), \eqref{e:|e''(i,e)|<Tr,Lip+} follow from (\ref{e:|e'(i)|<Tr,Lip}) and (\ref{e:|e'(e)|<Tr,Lip}) respectively and from the embedding result (\ref{embed}) recalled below which may be obtained from \cite[Th. 1.107]{triebel-2}. Indeed,  $\partial\Omega_{0}$ is an $(N-1)$-dimensional Lipschitz manifold and since we deal with exponents less than $1$, these embeddings carry over to Lipschitz manifolds.
\begin{multline}\label{embed}
{\forall t_{1},t_{2} \in [0,1],\; } \forall \,(p_{1},p_{2})\in[1,\infty],\\
 \;\;\Big[t_1-\frac{N-1}{p_1} > t_2 - \frac{N-1}{p_2}\; and\; 
 t_1 > t_2\Big] 
\quad
\Longrightarrow
\quad
W^{t_1,p_1}(\partial\Omega_0)\subset W^{t_2,p_2}(\partial\Omega_0).
\end{multline}

We apply this result with $t_{1}=s\in [0,1], p_{1}=2$ and,
$\big[t_{2}=1-1/r, p_{2}=r\big]$ to obtain \eqref{e:|e'(i,e)|<Tr,Lip++} and
$\big[t_{2}=1-1/(2r), p_{2}=2r\big]$ to obtain  \eqref{e:|e''(i,e)|<Tr,Lip+}.

For the proof of { ii) and iii)} we apply the same embeddings and iii), iv) from Theorem \ref{th:e',e''(i,intro)}. 
\end{proof}

\begin{remark} {The restrictions on $s_1,s_2$ in the Lipschitz case i) of Corollary \ref{c:e',e''(i,e,intro)} are coming from the restriction $p<\overline{p}_1=\overline{p}_1(\Omega_0)$ as indicated in Theorems \ref{th:hU',hU'',H12(i)}, \ref{th:hU',hU'',H12(e)}. We do not know the exact value of this $\overline{p}_1$, but we know that $\overline{p}_1>4$ if $N=2$ and $\overline{p}_1>3$ if $N\geq 3$. Thus (recall that $\overline{r}_1=\overline{p}_1/(\overline{p}_1-2)$), we may write
$\overline{r}_1=2-\sigma$ if $N=2$ 
and $\overline{r}_1=3-\sigma$ if $N\geq 3$, for some small $\sigma>0$.

In the range of $H^s$-spaces, the best estimates we obtain in the case i) of Corollary \ref{c:e',e''(i,e,intro)} may also be written as follows where the constants $C_{1},C_{2}$ here do not depend on $\xi\in\Theta$, but may depend on other variables, especially $\eps$):}

\begin{itemize}
\item 
if $N=2$, then there exists $\eps>0$ small such that 
\begin{eqnarray}
 |E'(\Om_{0})(\xi)| \leq C_1\|\xi\|_{H^{1/2-\eps}(\partial\Omega_0)},		\label{eq:lip1}\\
 |E''(\Om_{0})(\xi,\xi)| \leq C_2 \|\xi\|^2_{H^{1-\eps}(\partial\Omega_0)}.\label{eq:lip2}
\end{eqnarray}

\item
If $N=3$, then there exists $\eps>0$ small such that
\begin{eqnarray*}
|E'(\Om_{0})(\xi)| \leq C_1 \|\xi\|_{H^{1-\eps}(\partial\Omega_0)}.
\end{eqnarray*}
\end{itemize}
\end{remark}

\section{Estimates of shape derivatives}\label{s:estimates}

\subsection{Description of the method}\label{ss:method}

{We describe in this paragraph the main idea of the shape derivative estimates on the model example
$$
E(\Omega)=\int_\Omega K(U_{\Omega},\nabla U_\Omega):=\int_\Omega\frac{1}{2}|\nabla U_\Omega|^2,
$$
where $U_\Omega$ is solution of (\ref{e:LU=f(i)}) with $L=-\Delta$ and $\Omega,\Omega_0$ are as in (\ref{Omega}). The main point of our approach is that we will make {\em interior estimates} involving $W^{1,q}(\Omega_0)$-norms of the directions $\xi\in \Theta$ of differentiation. As already explained in the proof of Theorem \ref{th:e',e''(i,intro)}, we will then use the trace Lemma \ref{lem:extension} to obtain estimates in terms of $W^{1-1/q, q}(\partial\Omega_0)$-norms of $\xi$. By doing so, we avoid integrations by parts which would be impossible due to the poor regularity of the boundary of $\Omega_0$. Moreover, as we will se below, the shape derivative estimates will only depend on the regularity of $U_0:=U_{\Omega_0}$.

One first has to show that $E$ is differentiable.

By changing variable $x=(I+\theta)(y)$, for $\theta\in\Theta$ small, we have
$$\E(\theta)
=
\int_{\Omega_0}
\frac{1}{2}\nabla\hat{U}_\theta\cdot \hat{M}_\theta\cdot\nabla\hat{U}_\theta dx,
$$
where
\begin{equation}\label{Mtheta} 
\hat{U}_\theta=U_\theta\circ(I+\theta),\;\;\hat{M}_\theta=[I+\nabla\theta]^{-1}\,^{t}\![I+\nabla\theta]^{-1}{\rm det}[I+\nabla\theta],\;\; U_\theta=U_{\Omega_\theta}.
\end{equation}
}
Note that $\theta\mapsto \hat{M}_\theta$ is $C^\infty$ from $\Theta$ to $(L^\infty(\Omega_0))^{N\times N}$ near $\theta=0$, see \cite{HP,simon-1}.
Then the differentiability of $\theta\mapsto\E(\theta)$ fully depends on the differentiability of $\hat{U}_\theta$. As it is classical, under reasonable assumptions on $L$ and the data, by using the implicit function theorem, one can prove that $\theta\in\Theta\mapsto \hat{U}_\theta\in H^1_0(\Omega_0)$ is of class $C^m$ near $\theta=0$, which implies that $\E$ is of class $C^m$ near $\theta=0$.\\

Next, let us fix a direction $\xi\in\Theta$; we will denote by $(\cdot)'$ the derivatives with respect to $\theta$ in the direction $\xi$.
In general, $U_0$ is more regular than $H^1(\Omega_0)$ and we consider $p\in(2,\infty)$ such that 
$U_0\in W^{1,p}(\Omega_0)$ (see Propositions \ref{p:LU=f(i),Lip} and \ref{p:LU=f(i),conv} below). We may write:
\begin{eqnarray}
 \E'(\theta)(\xi)=
 \int_{\Omega_0}
 \nabla\hat{U}_\theta\cdot \hat{M}_\theta\cdot\nabla\hat{U}_\theta' +
 \frac{1}{2}\nabla\hat{U}_\theta\cdot \hat{M}'_\theta\cdot\nabla\hat{U}_\theta,			\label{e:e'(th)}
\end{eqnarray}
which at $\theta=0$ implies (see(\ref{Mtheta}))
\begin{eqnarray}
 |\E'(0)(\xi)|
& \leq&
 C
 \left(
 \left|
 \int_{\Omega_0}
 \nabla{U}_0\cdot \nabla\hat{U}_0'
 \right| 
 +
 \int_{\Omega_0}
 |\nabla U_0|^2|\nabla\xi|
 \right)										\nonumber\\
&\leq&
 C
 \left(
 \left|
 \int_{\Omega_0}
 \nabla{U}_0\cdot \nabla\hat{U}_0'
 \right| 
 +
\|U_0\|^2_{W^{1,p}(\Omega_0)}
\|\xi\|_{W^{1,p/(p-2)}(\Omega_0)}
\right).											\label{e:|e'(0)|<(1)}
\end{eqnarray}

The second term in the last estimate being satisfactory, it remains to estimate the first term which involves $\hat U'_{0}$.
Note first that it would not help much to use a too simple H\"older inequality like
\[
\left|
 \int_{\Omega_0}
 \nabla{U}_0\cdot \nabla\hat{U}_0'
 \right|
\leq
C
\|U_0\|_{H^1(\Omega_0)}\|\xi\|_{W^{1,\infty}(\Omega_0)},
\]
which uses only the starting information that $\hat{U}'_0$ is a linear continuous map from $\Theta$ to
$H^1(\Omega_0)$.

Nor is it appropriate to write the term $\hat{U}'_0$ in terms of the shape derivative $U'_0$, i.e.
$\hat{U}'_0=U'_0+\nabla U_0\cdot\xi$. Indeed, in such a case one would have
\begin{eqnarray*}
\int_{\Omega_0}
\nabla{U}_0\cdot \nabla\hat{U}_0'
&=& 
\int_{\Omega_0}
 \nabla{U}_0\cdot \nabla{U}_0'
 + 
 \nabla U_0\cdot D^2U_0\cdot\xi + \nabla U_0\cdot\nabla \xi\cdot\nabla U_0.
\end{eqnarray*}
But we would need here regularity for $D^2 U_{0}$ and it is not available for the case we are interested in. Even if in the {semi}-convex situation, we can get some significant information on the first derivative, it becomes quite more difficult for the second derivative (see however \cite{LNP} for some progress in this direction).

For these reasons, we proceed with the estimate of the term with $\hat{U}'_0$ by going back the state equation : $LU_\theta=f,\; U_\theta\in H^1_0(\Omega_\theta)$. Recall that we chose $L=-\Delta$ for simplicity here (but the ideas will be the same for general $L$). 
The weak form of $LU_\theta=f$  in $\Omega_\theta$ transported in $\Omega_0$ is
\begin{eqnarray}
\int_{\Omega_0}
\varphi \hat{L}_\theta\hat{U}_\theta dx
&=&
\int_{\Omega_0}\hat{f}_\theta\varphi dx,\;\; \textrm{ for all }\; \varphi\in H^1_0(\Omega_0),		\label{e:hLphi}
\end{eqnarray}
with $\hat{f}_\theta=f\circ(I+\theta)\det[I+\nabla\theta]$ and
\begin{equation}\label{e:hL(intro)}
\hat{L}_\theta:H^1_0(\Omega_0)\mapsto H^{-1}(\Omega_0),\quad
\hat{L}_\theta\phi=\nabla\cdot(\hat{M}_\theta\cdot\nabla\phi),
\end{equation}
where $\hat{M}_\theta$ is defined in (\ref{Mtheta}). Note that $\hat{L}_0=L$ and $(\hat{L}_\theta\hat U_{\theta})'=\hat{L}_\theta\hat{U}'_\theta+\hat{L}'_\theta\hat{U}_\theta$. By differentiating (\ref{e:hLphi}) with respect to $\theta$ in the direction $\xi$, we obtain
\begin{eqnarray}
\int_{\Omega_0}
\varphi \hat{L}_\theta \hat{U}'_\theta dx
&=&
\int_{\Omega_0}
\varphi(
\hat{f}'_\theta - \hat{L}'_\theta \hat{U}_\theta),			\label{e:hL'}
\end{eqnarray}
which at $\theta=0$ gives
\begin{eqnarray}
\int_{\Omega_0}
\nabla\hat{U}'_0\cdot\nabla\varphi
&=&
\int_{\Omega_0}
\varphi(
\hat{f}'_0 - \nabla\cdot(\hat{M}_0'\cdot\nabla U_0)),			\label{e:hL'0}
\end{eqnarray}
The estimate (\ref{e:|e'(0)|<(1)}) suggests to take $\varphi=U_0\in H^1_{0}(\Om_{0})$ in (\ref{e:hL'0}), which then yields
\begin{eqnarray}
\left|
\int_{\Omega_0}\nabla U_0\cdot\nabla \hat{U}_0'
\right|
&\leq&
C(\Omega_0,f)
\int_{\Omega_0}
(|U_0|+|\nabla U_0|^2)(|\xi|+|\nabla\xi|)				\nonumber\\
&\leq&
C(\Omega_0,f,p)(1+\|U_0\|_{W^{1,p}(\Omega_0)})\|U_{0}\|_{W^{1,p}(\Omega_0)}\|\xi\|_{W^{1,p/(p-2)}(\Omega_0)}.		\label{e:|hUDpK|<}
\end{eqnarray}
Thus, going back to (\ref{e:|e'(0)|<(1)}), we deduce
\begin{eqnarray}
 |\E'(0)(\xi)|
&\leq&
C(\Omega_0,U_0,f,p)
\|\xi\|_{W^{1,p/(p-2)}(\Omega_0)}.								\label{e:|e'(0)|<(2)}
\end{eqnarray}
For more general $K$ and $L$, the choice of $\varphi$ is more involved (see the proof of Theorem \ref{th:hU',hU'',H12(i)}), but the procedure is the same.

In the same spirit, that is to say by differentiating (\ref{e:e'(th)}) and (\ref{e:hL'}) at $\theta=0$, we can obtain a similar estimate for 
$|\E''(0)(\xi,\xi)|$. 
For more details see Sections \ref{ss:e'',int}, \ref{ss:e'',ext}.\\

\subsection{The trace lemma}
\begin{lemma}\label{lem:extension}
Let $\Om$ be as in (\ref{Omega}) and Lipschitz. Let also
${\displaystyle q\in[1,\infty], s\in\left(\frac{1}{q},1+\frac{1}{q}\right)}$, or $q=s=1$.  
Then there exists $C=C(s,q,\Om)$ such that, for every $\xi\in W^{1,\infty}(\Om,\R^N)\cap W^{s,q}(\Om,\R^N)$,
\begin{equation}\label{eq:inf}
\inf\Big\{\; \|\hat{\xi}\|_{W^{s,q}(\Om)}, \;\;\hat{\xi}\in W^{1,\infty}(\Om,\R^N),\;\hat{\xi}=\xi \;\textrm{ on }\;\partial\Om\;\Big\} \;\leq C\|\xi\|_{W^{s-\frac{1}{q},q}(\partial\Om)}.
\end{equation}
\end{lemma}

The main tool for the proof of this lemma is the following classical trace/extension Theorem:

\begin{theorem}\label{th:extension}
Let $\Om$ be a set with Lipschitz boundary, $q\in[1,\infty]$ and $s\in(\frac{1}{q},1+\frac{1}{q})$ {or $q=s=1$}. The trace mapping
$\Tr$ initially defined on $C^\infty(\overline{\Om})$ extends as a bounded linear operator from $W^{s,q}(\Om)$ to $W^{s-\frac{1}{q},q}(\partial\Om)$. Moreover there exists a bounded linear operator
$$\Ext:W^{s-\frac{1}{q},q}(\partial\Om)\to W^{s,q}(\Om)$$
such that $\Tr\circ\Ext=Id$ on  $W^{s,q}(\partial\Om)$.
\end{theorem}
{See \cite[Theorem 3.1]{JK-1} and the references therein for the case $q=s=1$, and for example \cite[Th1 p197]{JW} for the other cases.}

\noindent{\bf Proof of Lemma \ref{lem:extension}:} 
Let $\xi\in W^{1,\infty}(\Om,\R^N)\cap W^{s,q}(\Om,\R^N)$ (we drop the notation $\R^N$ in the following). We remark that $\xi_{|\partial\Om}\in W^{s-\frac{1}{q},q}(\partial\Om)$. Using the extension defined in Theorem \ref{th:extension}, we can define $\hat{\xi}:=\Ext(\Tr\,\xi)\in W^{s,q}(\Om)$ satisfying
\begin{equation}\label{eq:extension}
\|\hat{\xi}\|_{W^{s,q}(\Om)}\leq C \|\xi \|_{W^{s-\frac{1}{q},q}(\partial\Om)},
\end{equation}
where $C=C(s,q,\Om)$.

Therefore $\hat{\xi}-\xi \in W^{s,q}_{0}(\Om)$ and is therefore the limit  in $W^{s,q}(\Omega)$ of functions $\alpha_{n}\in C^\infty_{0}(\Om)$. In other words, $\hat{\xi}$ is the limit in $W^{s,q}(\Omega)$ of $\varphi_{n}=\xi+\alpha_{n}$. Clearly $\varphi_{n}\in W^{1,\infty}(\Om)\cap W^{s,q}(\Om)$ and 
$\Tr(\varphi_{n})=\Tr(\xi)$, so that
$$\inf\Big\{\; \|\hat{\xi}\|_{W^{s,q}(\Om)}, \;\;\hat{\xi}\in W^{1,\infty}(\Om)\cap W^{s,q}(\Om),\;\hat{\xi}=\xi \;on\;\partial\Om\;\Big\} \;\leq \|\varphi_{n}\|_{W^{s,q}(\Om)},$$
and letting $n\to\infty$ and using \eqref{eq:extension}, we obtain the first estimate \eqref{eq:inf}.
\qed

\subsection{Shape derivative estimates for an interior PDE}\label{ss:e'',int}
In this section we will prove estimates for the first and second order shape derivatives of the energy (\ref{e:E(i)}). As explained above, it will rely on estimates of $U_\Omega$.

\subsubsection{Existence, uniqueness and regularity of the solution $U_\Omega$}

\begin{proposition}\label{p:LU=f(i),Lip}
{Let $\Omega$ be as in (\ref{Omega}) and Lipschitz. Let $L$ be as in (\ref{e:LU=f(i)}) with the matrix $(a_{ij})$ satisfying (\ref{e:a(i)}),(\ref{e:Mij=Mji}).}
\begin{itemize}
\item[i)] Assume
$b_i, c\in L^\infty(\Omega)$ and $c\geq0$.
If $f\in H^{-1}(\Omega)$,  then there exists a unique solution $U_\Omega\in H^1_0(\Omega)$ of (\ref{e:LU=f(i)}).
{\item[ii)]
Assume
{$a_{ij}\in W^{1,\infty}(\Om)$, $b_i, c\in L^\infty(\Omega)$} and $c\geq0$.
Then there exists $p_1=p_{1}(\Om)$ satisfying [$p_1>4$ if $N=2$ and 
$p_1>3$ if $N\geq 3$], such that for every $p\in(p_{1}',p_1)$ (where $1/p_1+1/p_1'=1$) and every $f\in W^{-1,p}(\Omega)$ the problem
(\ref{e:LU=f(i)}) admits a unique solution $U_\Omega\in W^{1,p}_0(\Omega)$ and
\begin{eqnarray}
\|U_\Omega\|_{W^{1,p}(\Omega)}
&\leq&
C(\Omega)\|f\|_{W^{-1,p}(\Omega)}.			\label{e:|u|W1p<C|f|W-1p}
\end{eqnarray}
}
\end{itemize}
\end{proposition}
\begin{proof}
The point i) is standard, { see \cite{GT}}. The sharp regularity result in ii)  has been first proved for $L=-\Delta$ in \cite[Thm. 0.5]{JK-1}. 
The complete proof of ii) is based on \cite[Theorem C]{shen} and Remark \ref{r:shen->W1p} for $Lu=-\sum_{i,j}\partial_i(a_{ij}\partial_ju)$,
and Remark \ref{r:JLW->L} for $L$ of general form. {See also Remark \ref{r:ii)+mitrea} for a different proof with a slightly stronger regularity on $a_{ij}$. Note that, according to \cite{shen}, the same result is
valid with quite weaker regularity on $a_{ij}$, like asking that they be in
$\mbox{VMO}(\R^N)$.}
\end{proof}

\begin{remark}{\bf Notation:} \label{r:p1*,L*}
{If $L$ satisfies the assumptions of ii) above, then so does its adjoint $L^*$ as we can easily see by writing 
$(b_iu)_{x_i}=(b_i)_{x_i}u+b_iu_{x_i}$, where $b_i\in W^{1,\infty}(\Omega)$. } Let $p_1$, $p_1'$ (resp. $p_1^*$, $p_1'^*$) be the numbers 
associated to the equation $LU=f$, $U\in W^{1,p}_0(\Omega)$ (resp. $L^*U=f$, $U\in W^{1,p}_0(\Omega)$) as given by ii), Proposition \ref{p:LU=f(i),Lip}. 
Then we set $$\overline{p}_1:=\min\{p_1,p_1^*\},\;\;\overline{p}'_1:=\max\{p'_1,{p'}_1^*\}.$$
{Note that, if $\Omega$ is Lipschitz, then $\overline{p}_1>4$ if $N=2$ and $\overline{p}_1>3$ if $N\geq 3$.}
\end{remark}

\begin{proposition}\label{p:LU=f(i),conv}
{Let $\Omega$ be as in (\ref{Omega}) and {semi}-convex. Let $L$ be as in (\ref{e:LU=f(i)}) with the matrix $(a_{ij})$ satisfying (\ref{e:a(i)}),(\ref{e:Mij=Mji}).}
\begin{itemize}
\item[i)]
Assume $a_{ij}\in W^{1,\infty}(\Omega), b_i, c\in L^\infty(\Omega)$ and $c\geq0$.
Then for every $p\in (1,\infty)$ and for all $f\in W^{-1,p}(\Omega)$, the problem (\ref{e:LU=f(i)}) admits a unique solution $U_\Omega\in W^{1,p}_0(\Omega)$ and (\ref{e:|u|W1p<C|f|W-1p}) holds.
\item[ii)]
Assume $a_{ij}\in W^{1,\infty}(\Omega)$, $b_i, c\in L^\infty(\Omega)$ and $c\geq 0$.
Then for every $f\in L^\infty(\Omega)$ there exists a unique solution 
$U_\Omega\in H^1_0(\Omega)\cap W^{1,\infty}(\Omega)$~of~(\ref{e:LU=f(i)}).
\end{itemize}
\end{proposition}

\begin{proof}
The result in i) was proven for $L=-\Delta$, see \cite[Corollary 1 and Remark]{fromm-1}. {For general $L$, we refer to \cite{JLW} as explained in  Remark \ref{r:dR->qc} and we use Remark \ref{r:JLW->L} below to include the first order terms.}
For ii), as $f\in L^\infty(\Omega)$, from i) it follows that $U_\Omega\in W^{1,p}(\Omega)$
for all $p\in(1,\infty)$, which implies $U_\Omega\in C^0(\overline{\Om})$.
Furthermore, applying \cite[Thm. 8.8]{GT}, it follows $u\in W^{2,2}_{loc}(\Omega)$. Applying \cite[Theorem 9.13]{GT} (locally in $\Omega$, with $T=\emptyset$) and then \cite[Lemma 9.16]{GT}, we get $u\in W^{2,q}_{loc}(\Omega)$, for all $q>2$.
From Sobolev embeddings it follows $u\in C^0(\overline{\Omega})\cap C^2(\Omega)$. Then ii) follows from \cite[Thm. 15.9]{GT}.
\end{proof}

\begin{remark}\label{r:shen->W1p}
Theorem C in \cite{shen} states that if $a_{ij}\in \mbox{VMO}(\mathbb R^N)$ {(and in particular if $a_{ij}\in W^{1,\infty}(\Om)$)}, and $L=-\partial_i(a_{ij}(x)\partial_j)$, then
$T:=\nabla (L)^{-1/2}$ is continuous from $L^p(\Omega)$ to $(L^p(\Omega))^N$, for $p\in(1,p_1)$, with 
$p_1=p_1(N)$, $p_1(2)>4$, $p_1(N)>3$ for $N\geq 3$.
Then it follows that $L$ is continuous invertible from $W^{1,p}_0(\Omega)$ onto $W^{-1,p}(\Omega)$ for 
$p\in(p_1',p_1)$. Indeed, first note that this is equivalent to show that $\nabla (L)^{-1}\mbox{\rm div}$ is continuous from
$(L^p(\Omega))^N$ to $L^p(\Omega)$. Next, we note that from \cite[Theorem C]{shen}, $T^*=(L)^{-1/2}\mbox{\rm div}$ is continuous from 
$(L^q(\Omega))^N$ to $L^q(\Omega)$,
for $q\in(p_1',\infty)$. Then the desired continuity for $\nabla(L)^{-1}\mbox{div}$ follows from the fact that 
$\nabla(L)^{-1}\mbox{div}=TT^*$.
\end{remark}

{\begin{remark}\label{r:JLW->L}
To complete the proof of ii) in Proposition \ref{p:LU=f(i),Lip} and of i) in Proposition \ref{p:LU=f(i),conv} , let us check that if the map
$u\to Au:=-\partial_i(a_{ij}(x)\partial_ju)$ defines an isomorphism from $W^{1,p}_0(\Omega)$ into $W^{-1,p}(\Omega)$, then so does $L$.
Indeed, let $B=L-A$, {$Bu=b_i\partial_iu+ cu$}, and consider the equation 
\begin{equation}\label{e:LinW1p}
Lu=f,\quad u\in W^{1,p}_0(\Omega),\quad f\in W^{-1,p}(\Omega),
\end{equation}
which is equivalent to $Au+Bu=f$. Multiplying this equation by $A^{-1}$, we get the equivalent equation
\begin{equation}\label{e:LinLp}
(I+K)u = g,\quad
g=A^{-1}f\in W^{1,p}_0(\Omega),\quad
u\in W_0^{1,p}(\Omega),
\end{equation}
where $K=A^{-1}B: W_0^{1,p}(\Omega) \mapsto W^{1,p}_0(\Omega)$ is compact since $B$ is compact from $W_0^{1,p}$ into $W^{-1,p}$ and $A^{-1}$ is continuous from $W^{-1,p}$ into $W_0^{1,p}$.

Then $I+K$ satisfies the Fredholm alternative (see e.g. \cite[Thm. 6.6]{brezis}). Furthermore, $Ker(I+K)=Ker(L)=\{0\}$. Indeed, assume $u\in W_0^{1,p}(\Omega)$, $Lu=0$. If $p\geq 2$, we can directly use the uniqueness result in \cite[Thm. 8.1 and Cor. 8.2]{GT} to deduce $u=0$.  If $p\in (1,2)$, we obtain the same conclusion by using that $L^*: W_0^{1,p'}\mapsto W^{-1,p'}$ is onto which follows from \cite[Thm.8.3 and 8.6]{GT} (see also the remark following Corollary 8.7 in \cite{GT}).

Then, Fredholm alternative \cite[Thm. 6.6]{brezis} implies that $I+K$ defines an isomorphism from $W_0^{1,p}(\Omega)$ onto itself.
\hfill$\Box$
\end{remark}
}

\begin{remark}\label{r:ii)+mitrea}
{Different approaches may be quoted for the last point of Proposition \ref{p:LU=f(i),Lip}.
They only require a slightly stronger regularity hypothesis (like
$a_{ij}\in C^{1+\gamma}(\Omega), \gamma\in (0,1)$), but they are also quite
interesting. Let us just recall the idea:}
we need to focus only in the case $L=-\partial_i(a_{ij}(x)\partial_i)$, because the case of general $L$ follows from Remark \ref{r:JLW->L}.
The case $N=2$ follows from \cite{MM}.
For the case $N\geq3$ we consider the manifold $(\Omega,g)$, where $g$ is the metric given by $g=\mbox{det}(a_{ij})^{1/(N-2)}A^{-1}$. Then 
(from \cite[p. 186]{MT1} and the formula of Laplace-Beltrami operator, see \cite[(1.1)]{MM} for instance), the equation $Lu=f$, $u\in W^{1,p}_0(\Omega)$ with $f\in W^{-1,p}(\Omega)$,
becomes $-\Delta_g u = h$, with $h=\mbox{det}(a_{ij})^{1/(2-N)}f\in W^{-1,p}(\Omega)$. Note that $g\in C^{1,1}(\Omega)$ and 
$W^{-1,p}(\Omega)$ norms of $f$ and $h$ are equivalent. Then the claim ii) follows from \cite[Corollary 13.2]{MT}.\\
\end{remark}

\subsubsection{Estimates of the shape derivatives of the solution}
The results of this section give estimates for $\hat{U}'_0$ and $\hat{U}''_0$. 
Note that while for $\hat{U}'_0$ we prove an $H^1(\Omega_0)$-estimate, we also prove more involved estimates for 
$\hat{U}'_0$ and $\hat{U}''_0$. These estimates are  motivated by the functional $K$ involved in the energy (\ref{e:E(i)}), see also Theorem \ref{th:|e'|+|e''|(i)}.
\begin{theorem}\label{th:hU',hU'',H12(i)}
{Let $\Omega_0$ be as in (\ref{Omega}) and let $L$ be as in (\ref{e:LU=f(i)}) with the matrix $(a_{ij})$ satisfying (\ref{e:a(i)}),(\ref{e:Mij=Mji}).} 
Let $\theta\in\Theta$, $\Omega_\theta=(I+\theta)(\Omega_0)$,
$U_\theta\in H^1_0(\Omega_\theta)$ solution of (\ref{e:LU=f(i)}) (see  i) in Proposition \ref{p:LU=f(i),Lip}),
$\hat{U}_\theta=U_\theta\circ(I+\theta)$.
\begin{itemize}
\item[i)] If $a_{ij}, b_i, c, f\in W^{m,\infty}({B_{R}})$, $m\in\mathbb N^*$, then $\hat{U}_\theta\in C^m(\mathcal{V},H^1_0(\Omega_0))$ where $\mathcal{V}$ is a neighborhood of $\theta=0$ in $\Theta$.
\end{itemize}
{\it In the following, we assume the previous hypotheses are satisfied {for $m=2$. }}
\begin{itemize}
\item[ii)]
If moreover $\Omega_0$ is Lipschitz 
{
and $\alpha_{ij}\in W^{1,\infty}(\Omega)$,
$\beta_i, \gamma,\alpha_{00}\in L^\infty(\Omega)$
}
then for all $p\in(2,\ol{p}_1)$ we have
\begin{eqnarray}
\hspace*{-5mm}
\|\hat{U}'_0\|_{H^1(\Omega_0)}
&\leq&
 C _1
  \|\xi\|_{W^{1,2p/(p-2)}(\Omega_0)},						\label{e:|hU'|H1<(i)}\\
\hspace*{-5mm}
\Big|
 \int_{\Omega_0}
\partial_U K(\cdot,U_0,\nabla U_0) \hat{U}'_0
+
\partial_q K(\cdot,U_0,\nabla U_0)\cdot  \nabla\hat{U}'_0
 \Big|		
&\leq&
 C_1
 \|\xi\|_{W^{1,p/(p-2)}(\Omega_0)},							\label{e:|hU'phi|<(i)}\\
\hspace*{-5mm}
\Big|
 \int_{\Omega_0}
\partial_U K(\cdot,U_0,\nabla U_0) \hat{U}''_0
 +
 \partial_q K(\cdot,U_0,\nabla U_0)\cdot  \nabla\hat{U}''_0
 \Big|
&\leq&
 C_2
 \|\xi\|^2_{W^{1,2p/(p-2)}(\Omega_0)},		\label{e:|hU''phi|<(i)}
\end{eqnarray}
where
$\ol{p}_{1}$ is defined in Remark  \ref{r:p1*,L*} ($\overline{p}_1>4$ if $N=2$, $\overline{p}_1>3$ if $N\geq 3$), 
$\hat{U}'_0=\hat{U}_\theta'(0)(\xi)$, 
$\hat{U}''_0=\hat{U}_\theta''(0)(\xi,\xi)$,
$C_i= C_i(\Omega_0,L,f,K,p)$.
\end{itemize}
\end{theorem}
\begin{proof}\\
{\it Step 1. } 
The proof of i) is classical. It is based on the implicit function theorem.

After changing the variable $x=(I+\theta)(y)$, the weak form of (\ref{e:LU=f(i)}) in $\Omega_\theta$ is transformed in $\Omega_0$ as follows
\begin{eqnarray}
 \langle
\hat{L}_\theta\hat{U}_\theta,\varphi\rangle_{H^{-1}(\Om_{0})\times H^1_{0}(\Om_{0})}
&=&
\int_{\Omega_0}
\varphi \hat{f}_\theta,\quad\forall \varphi\in H^1_0(\Omega_0),	\label{e:hLhU=hf,w(i)}
\end{eqnarray}
where  $\hat{L}_\theta:H^1_0(\Omega_0)\mapsto H^{-1}(\Omega_0)$ is defined by
\begin{eqnarray}
\langle
\hat{L}_\theta\hat{U},\varphi\rangle
&=&
\langle
\partial_l(\hat{a}_{kl}\partial_k\hat{U})
 +
\hat{b}_k\partial_k\hat{U} 
 +
 \hat{c}\hat{U}
,
\varphi\rangle					\nonumber\\
&=&
\int_{\Omega_0}
 \hat{a}_{kl}\partial_k\hat{U}\partial_l{\varphi} 
 +
 (\hat{b}_k\partial_k\hat{U} 
 +
 \hat{c}\hat{U}
)\varphi,	\;\;\;\textrm{ for }(\hat{U},\varphi)\in H^1_{0}(\Om_{0})^2,												\label{e:hLhU(i)}
\end{eqnarray}
and
\[
\begin{array}{rclccrcl}
\hat{a}_{kl}&=&\sum_{i,j}a_{ij}\circ(I+\theta){T}_{ki}{T}_{lj}{\rm det}[I+\nabla\theta],
&&
\hat{b}_k&=&\sum_ib_i\circ(I+\theta){T}_{ki}{\rm det}[I+\nabla\theta],\\[1mm]
\hat{c}&=&c\circ(I+\theta){\rm det}[I+\nabla\theta],
&&
\hat{f}&=&f\circ(I+\theta){\rm det}[I+\nabla\theta],\\[1mm]
T&=&(T_{kl})=[Id+\nabla\theta]^{-1}.
\end{array}
\]
Consider $F$ defined by
\[
 \begin{array}{lclll}
 F:&\Theta\times H^1_0(\Omega_0)&\longrightarrow&H^{-1}(\Omega)\\
    &(\theta,\hat{U})&\longmapsto&
\left[\varphi\mapsto  \langle
 \hat{L}_\theta\hat{U}-\hat{f}, 
 \varphi\rangle\right].
\end{array}
\]
It is easy to check that $F$ is well defined and of class $C^m$ in a neighborhood of $(0,U_0)\in\Theta\times H^1_0(\Omega_0)$, and 
\[
 \partial_{{U}}F(0,U_0)(U):\varphi\mapsto
 \langle
 LU,\varphi\rangle
 =
 \int_{\Omega_0}
 a_{ij}\partial_i U\partial_j \varphi 
 +
 (b_i\partial_i U + c U)\varphi,
\quad\forall U, \varphi\in H^1_0(\Omega_0).
\]
Note that from Proposition \ref{p:LU=f(i),Lip}, $\partial_{{U}}F(0,U_0)$ defines an isomorphism from $H^1_0(\Omega_0)$ to $H^{-1}(\Omega_0)$. Then, from implicit function Theorem there exists a $C^m$ map, $\hat{U}:\theta\mapsto\hat{U}(\theta)$, such that $F(\theta,\hat{U}(\theta)) =0$ for $\|\theta\|_\Theta$ small.
From the uniqueness of solution to (\ref{e:LU=f(i)}) we obtain $\hat{U}_\theta=\hat{U}(\theta)$, which proves the regularity of $\theta\mapsto\hat{U}_\theta$.

\noindent
{\it Step 2. }
We prove the estimates (\ref{e:|hU'|H1<(i)})-(\ref{e:|hU''phi|<(i)}) by differentiating (\ref{e:hLhU=hf,w(i)}) with respect to $\theta$. The differentiation is allowed because $\hat{U}_\theta$ is  differentiable and $a_{ij}, b_i, c, f$ are regular enough.

Differentiating (\ref{e:hLhU=hf,w(i)}) once with respect to $\theta$ gives
\begin{eqnarray}
\langle
 \hat{L}_\theta\hat{U}'_\theta,\varphi\rangle
&=&
\langle
\hat{f}'_\theta
-
\hat{L}'_\theta\hat{U}_\theta,\varphi\rangle.								\label{e:(LU-f)'=0(i)-1}
\end{eqnarray}
We take $\theta=0$ in   (\ref{e:(LU-f)'=0(i)-1}). {Note that $\hat{L}_0=L$}. By isolating all the terms with $\hat{U}'_0$ we obtain
\begin{eqnarray}
\langle
 L\hat{U}'_0,{\varphi}\rangle
&=&
\langle
 L^*{\varphi},\hat{U}'_0\rangle						\nonumber\\   
&=&
\int_{\Omega_0}
 \hat{f}'_0{\varphi} 
-
\left(
\hat{a}_{kl}'\partial_k U_0\partial_l {\varphi}
+
(\hat{b}_k'\partial_k U_0
+
\hat{c}'_0 U_0)\varphi
\right).						\label{e:hU'L*}
\end{eqnarray}
We now choose suitable test functions $\varphi$ to prove (\ref{e:|hU'|H1<(i)}) and (\ref{e:|hU'phi|<(i)}).
Let ${\varphi}\in H^1_0(\Omega_0)$ be the solution of
\begin{equation}\label{e:phi1,U'(i)}
L^*{\varphi}=\hat{U}'_0-\Delta\hat{U}'_0=:g^*. 
\end{equation}
Note that $g^*\in H^{-1}(\Omega_0)$ and then the solution
${\varphi}$ is uniquely defined in $H^1_0(\Omega_0)$ and satisfies (see \cite{LU-1})
\begin{equation}\label{e:|phi1|H1<(i)}
\|{\varphi}\|_{H^1(\Omega_0)}
\leq 
 C\|g^*\|_{H^{-1}(\Omega_0)}
 \leq
 C\|\hat{U}'_0\|_{H^1(\Omega_0)},\quad C=C(\Omega_0,L).
\end{equation}
Then from (\ref{e:hU'L*}) it follows
\begin{eqnarray}
\|\hat{U}'_0\|_{H^1(\Omega_0)}^2
&\leq&
C_1
\int_{\Omega_0}
(1+|U_0|+|\nabla U_0|)
(|\xi|+|\nabla\xi|)
(|{\varphi}|+|\nabla{\varphi}|)			  \nonumber	\\
&\leq&
C_1
(1+  \|U_0\|_{W^{1,p}(\Omega_0)})\|\xi\|_{W^{1,2p/(p-2)}(\Omega_0)}\|\varphi\|_{H^1(\Omega_0)},
									\label{e:|hU'|H1^2<}
\end{eqnarray}
where we applied H\"older inequality with three terms, and $C_1=C_1(\Omega_0,L,f)$, $p\in(2,p_1)$, 
{$p_1=p_1(\Om_{0})$ as given by Proposition \ref{p:LU=f(i),Lip}
for $L$}. Together with (\ref{e:|phi1|H1<(i)}) this proves (\ref{e:|hU'|H1<(i)}).

Now we consider ${\varphi}\in H^1_0(\Omega_0)$ the solution of
\begin{eqnarray}
L^*{\varphi}
&=&
\partial_U K(\cdot,U_0,\nabla U_0) - \nabla\cdot \partial_q K(\cdot,U_0,\nabla U_0)		\nonumber	\\
&=&
(\alpha_{00}U_0 + \gamma)
-
\nabla\cdot
\left(
\alpha\nabla U_0
+
\beta
\right)=:
g^*.							\label{e:phi2,U'(i)}
\end{eqnarray}
Let $q\in (2,\overline{p}_1)$. 
Recall that $U_0\in W^{1,q}(\Omega_0)$ by Proposition \ref{p:LU=f(i),Lip}. From \eqref{e:K(i)}, we deduce
\[
\partial_U K(\cdot,U_0,\nabla U_0), \partial_q K(\cdot,U_0,\nabla U_0)\in L^{{q}}(\Omega_0).
\]
It follows that $g^*\in W^{-1,q}(\Omega_0)$. By Proposition \ref{p:LU=f(i),Lip} applied to $L^*$, the solution
${\varphi}$ of (\ref {e:phi2,U'(i)}) is uniquely defined in $W_0^{1,q}(\Omega_0)$ and satisfies
\begin{equation}\label{e:|phi2|H1<(i)}
\|{\varphi}\|_{W^{1,q}(\Omega_0)}
\leq 
 C\|g^*\|_{W^{-1,q}(\Omega_0)}
 =
 C,\quad C=C(\Omega_0,L,K,U_0,q).
\end{equation}
Then from (\ref{e:hU'L*}) it follows
\begin{eqnarray}
\Big|
 \int_{\Omega_0}
\partial_U K(\cdot,U_0,\nabla U_0) \hat{U}'_0
&+&
\partial_q K(U_0,\nabla U_0)\cdot  \nabla\hat{U}'_0
 \Big|		\nonumber\\
&=&
\Big|
\langle
L^*\varphi,\hat{U}'_0\rangle\Big|
=
\langle
L\hat{U}',\varphi
\rangle\Big|
\nonumber\\
 &\leq&
C_1
\int_{\Omega_0}
(1+|U_0|+|\nabla U_0|)
(|\xi|+|\nabla\xi|)
(|{\varphi}|+|\nabla{\varphi}|)			  \nonumber	\\
&\leq&
C_1
(1+\|U_0\|_{W^{1,p}(\Omega_0})\|\varphi\|_{W^{1,q}(\Omega_0)}\|\xi\|_{W^{1,1/(1-1/p-1/q)}(\Omega_0)}),
\label{e:|hU'phi|<(i)-2}
\end{eqnarray}
with $C_1=C_1(\Omega_0,L,K,f,q)$. Here we chose $q:=p$ where $p$ is given in ii) of Theorem \ref{th:hU',hU'',H12(i)} and the estimate (\ref{e:|hU'phi|<(i)}) follows. (Note that the use in considering $q\neq p$ in the previous computations will appear later in the case of { {semi}-convex domains}).

\noindent
{\it Step 3. }
Differentiating (\ref{e:(LU-f)'=0(i)-1}) at $\theta=0$ and isolating the terms with $\hat{U}''_0$ gives
\begin{eqnarray}
\langle
L\hat{U}''_0,\varphi\rangle
&=&
\int_{\Omega_0}
\hat{f}''_0\varphi
-
(\hat{a}_{kl}''\partial_k\hat{U}_0\partial_l\varphi
+
(\hat{b}_k''\partial_k\hat{U}_0 + \hat{c}''_0\hat{U}_0) \varphi
)
\nonumber\\
&&-
\int_{\Omega_0}
2
\left(
\hat{a}_{kl}'\partial_k\hat{U}'_0\partial_l \varphi
+
(\hat{b}_k'\partial_k\hat{U}'_0
+
\hat{c}'_0\hat{U}'_0
)
\varphi 
\right).					\label{e:(LU-f)''=0(i)-1}
\end{eqnarray}
It implies
\begin{eqnarray}
\hspace*{-7mm}
\left|
\langle L^*\varphi,\hat{U}''_0\rangle
\right|					
&\leq&
C_2
\int_{\Omega_0}
\Big(
(1+|U_0|+|\nabla U_0|)(|\xi|^2+|\nabla\xi|^2) +	
(|\hat{U}'_0|+|\nabla\hat{U}'_0|)(|\xi|+|\nabla\xi|)
\Big)	
(|{\varphi}|+|\nabla{\varphi}|),							\label{e:|hU''phi|<(i)-1}
\end{eqnarray}
where $C_2=C_2(\Omega_0,L,f)$. 

Then (\ref{e:|hU''phi|<(i)-1}) with $\varphi$ solution of (\ref{e:phi2,U'(i)}), together with (\ref{e:|hU'|H1<(i)})  yield
\begin{eqnarray}
\hspace*{-8mm}
\left|\langle L^*\varphi,\hat{U}''_0\rangle\right|	
&=&
\left|
\int_{\Omega_0}
\partial_U K(\cdot,U_0,\nabla U_0)\hat{U}''_0
+
\partial_q K(\cdot,U_0,\nabla U_0)\cdot \nabla\hat{U}''_0
\right|
\nonumber\\
&\leq&	
C_2
\Big(
(1+\|U_0\|_{W^{1,p}(\Omega_0)})\||\xi|^2+|\nabla\xi|^2\|_{L^{1/(1-1/p-1/q)}(\Omega_0)} 
+
\|\nabla\hat{U}'_0\|_{L^2(\Omega_0)}\|\xi\|_{W^{1,{1/(1/2-1/q)}}(\Omega_0)}\Big)
\nonumber\\
&&
\hspace*{5mm}
{\|\varphi\|_{W^{1,q}(\Omega_0)}}			\nonumber\\
&\leq&
C_2
(1+\|U_0\|_{W^{1,p}(\Omega_0)})\|U_0\|_{W^{1,q}(\Omega_0)}	\nonumber\\
&&
\hspace*{5mm}
\Big(
\|\xi\|^2_{W^{1,2/(1-1/p-1/q)}(\Omega_0)}  
+
\|\xi\|_{W^{1,2p/(p-2)}(\Omega_0)}\|\xi\|_{W^{1,2q/(q-2)}(\Omega_0)}
\Big),							\label{e:|hU''phi|<(i)-2}
\end{eqnarray}
with $C_2=C_2(\Omega_0,L,f,K,p,q)$. Again we choose $q:=p$ and this proves (\ref{e:|hU''phi|<(i)}).
\end{proof}

\begin{proposition}\label{p:hU',hU'',H12(i),conv}
{Besides the assumptions of ii) in Theorem  \ref{th:hU',hU'',H12(i)},} we assume $\Omega_0$ is {semi}-convex.  Then the following holds.
\begin{itemize}
\item[i)]
For all $p\in(2,\infty)$ we have
\begin{eqnarray}
\|\hat{U}'_0\|_{H^1(\Omega_0)}
&\leq&
 C _1
 \|\xi\|_{H^1(\Omega_0)},						\label{e:|hU'|H1<(i),conv}\\
\Big|
 \int_{\Omega_0}
\partial_U K(\cdot,U_0,\nabla U_0) \hat{U}'_0
 +
\partial_q K(\cdot,U_0,\nabla U_0)\cdot  \nabla\hat{U}'_0
 \Big|		
&\leq&
 C_1
 \|\xi\|^2_{W^{1,p/(p-2)}(\Omega_0)},					\label{e:|hU'phi|<(i),conv}\\
\Big|
 \int_{\Omega_0}
\partial_UK(\cdot,U_0,\nabla U_0) \hat{U}''_0 
 +
\partial_q K(\cdot,U_0,\nabla U_0)\cdot \nabla\hat{U}''_0
 \Big|
 &\leq&
 C_2
\|\xi\|^2_{W^{1,2p/(p-2)}(\Omega_0)},						\label{e:|hU''phi|<(i),conv}
\end{eqnarray}
where  $C_i=C_i(\Omega_0,L,f,K,p)$.
\item[ii)]
If furthermore $L$ is self-adjoint and $E$ is its energy associated, i.e. $\partial_U K(\cdot,U,\nabla U) - \nabla\cdot\partial_q K(\cdot,U,\nabla U)=LU$,
then we can take $p=\infty$ in (\ref{e:|hU'phi|<(i),conv}), (\ref{e:|hU''phi|<(i),conv}).
\end{itemize}
\end{proposition}
\begin{proof}
{We proceed as in Theorem \ref{th:hU',hU'',H12(i)} using the extra property that $\Omega_0$ is {semi}-convex. We now have $U_0\in  W^{1,\infty}(\Omega_0)$ (see Proposition \ref{p:LU=f(i),conv}). Moreover, the solution $\varphi$ of (\ref{e:phi2,U'(i)}) satisfies $\varphi\in W^{1,q}(\Omega)$ for all $q\in (2,\infty)$.
Then we can take $p=\infty$ in (\ref{e:|hU'|H1^2<}) which implies (\ref{e:|hU'|H1<(i),conv}).
Furthermore,   we can apply  (\ref{e:|hU'phi|<(i)-2}) and (\ref{e:|hU''phi|<(i)-2}) with $p=\infty$  and $q\in(2,\infty)$ arbitrary, which prove 
 (\ref{e:|hU'phi|<(i),conv}) and (\ref{e:|hU''phi|<(i),conv}).\\
In the case $\Omega$ { semi}-convex, $L$ self-adjoint and $K$ the energy associated to $L$, then 
$L=L^*$ and {$L^*\varphi=L U_{0}$}. Therefore, $\varphi=U_0$ in 
(\ref{e:phi2,U'(i)}). Hence $U_0\in W^{1,\infty}(\Omega_0)$ and $\varphi\in W^{1,\infty}(\Omega_0)$. Then we proceed as 
above with $p=q=\infty$.
}
\end{proof}

\begin{remark}
The results of Proposition \ref{p:hU',hU'',H12(i),conv} hold for any $\Omega_0$ such that
$U_0\in W^{1,\infty}(\Omega_0)$ and such that $\varphi\in W^{1,q}_0(\Omega_0)\mapsto L^*\varphi\in W^{-1,q}(\Omega_0)$ is bounded and 
invertible for all $q\in (2,\infty)$. 
\end{remark}

\begin{remark}\label{r:|hU''|H1<(i)}
In Theorem \ref{th:hU',hU'',H12(i)}, we could try to estimate $\|\hat{U}''_0\|_{H^1(\Omega_0)}$.
Indeed, in step 3 of Theorem  \ref{th:hU',hU'',H12(i)}, we take ${\varphi}$ to be the solution of $L^*{\varphi}=\hat{U}''_0-\Delta\hat{U}''_0=:f^*$, so that 
$\|\hat{U}''_0\|^2_{H^1(\Omega_0)}=\langle L^*\varphi,\hat{U}''_0\rangle$.
Then we can proceed as in (\ref{e:|hU''phi|<(i)-2}). However, as we have only $\hat{U}'_0, \hat{U}''_0\in H^1_0(\Omega_0)$,  then $f^*\in H^{-1}(\Omega_0)$ and so $\varphi\in H^1_0(\Omega_0)$ (sharp in general).
Therefore, in (\ref{e:|hU''phi|<(i)-2}) we have $q=2$ and it leads to
\begin{eqnarray}
\|\hat{U}''_0\|_{H^1(\Omega_0)}^2
&\leq&
C_2
(1+\|U_0\|_{W^{1,p}(\Omega_0)})\|U_0\|_{W^{1,2}(\Omega_0)}	\nonumber\\
&&
\Big(
\|\xi\|^2_{W^{1,4p/(p-2)}(\Omega_0)} 
+	
\|\xi\|_{W^{1,2p/(p-2)}(\Omega_0)}\|\xi\|_{W^{1,\infty}(\Omega_0)}
\Big),					\label{e:|hU''|H1<(i)}									
\end{eqnarray}
which is not appropriate as it contains the strong norm $\|\xi\|_{W^{1,\infty}(\Omega_0)}$.
\end{remark}

\subsubsection{Estimates of the shape derivatives of the energy}
\begin{theorem}\label{th:|e'|+|e''|(i)}
{Let $\Omega_0$ be as in (\ref{Omega}) and let $L$ be as in (\ref{e:LU=f(i)}) with the matrix $(a_{ij})$ satisfying (\ref{e:a(i)}),(\ref{e:Mij=Mji}).} The following properties hold.
\begin{itemize}
\item[i)]
If $a_{ij}, b_i, c, f,\ \alpha_{ij}, \beta_i, \gamma, \delta \in W^{m,\infty}({B_{R}})$, $m\geq 1$, 
then $\E:\theta\mapsto\E(\theta)$ is of class $C^m$ in a neighborhood of $\theta=0$ in $\Theta$.
\end{itemize}
{\it In the following, we assume the previous hypotheses are satisfied {for $m=2$. }
\begin{itemize}
\item[ii)]
{If $\Omega_0$ is Lipschitz,} then for all $p\in(2,\ol{p}_1)$ we have
\begin{eqnarray}
 |\E'(0)(\xi)|
 &\leq& 
 C_1
\|\xi\|_{W^{1,p/(p-2)}(\Omega_0)},			\label{e:|e'(i)|<,Lip}
\\
 |\E''(0)(\xi,\xi)|
 &\leq& 
 C_2
\|\xi\|^2_{W^{1,2p/(p-2)}(\Omega_0)},				\label{e:|e''(i)|<,Lip}
\end{eqnarray}
where $C_i=C_i(\Omega_0,L,f,K,p)$ and $\ol{p}_1$ (introduced in Remark \ref{r:p1*,L*}) satisfies $\overline{p}_1>4$ if $N=2$ and $\overline{p}_1>3$ if $N\geq 3$. 
\item[iii)] 
If  $\Omega_0$ is {semi}-convex,  then (\ref{e:|e'(i)|<,Lip}) and  (\ref{e:|e''(i)|<,Lip}) hold for all $p\in(2, \infty)$.
\item[iv)]
{
If $\Omega_0$ is {semi}-convex, $L$ is self-adjoint and $K$ is its associated energy, then (\ref{e:|e'(i)|<,Lip}) and  (\ref{e:|e''(i)|<,Lip}) hold for $p=\infty$.
}
\end{itemize}}
\end{theorem}
\begin{remark}\label{rk:Lp}
Theorem \ref{th:|e'|+|e''|(i)} can be stated in a more general form, based on the regularity of the state solution related to the operator $L$ and its adjoint $L^*$.
Namely, if
\begin{itemize}
\item[a)]
$a_{ij}, b_i, c, f,\ \alpha_{ij}, \beta_i, \gamma, \delta \in W^{2,\infty}(\Omega_0)$, 
\item[b)]
the operators $L$ and $L^*$ define isomorphisms from $W^{1,p}_0(\Omega)\mapsto W^{-1,p}(\Omega)$, for some $p\in[1,\infty]$
\end{itemize}
then (\ref{e:|e'(i)|<,Lip}) and  (\ref{e:|e''(i)|<,Lip}) hold for this $p$. 
\end{remark}

\begin{proof}{\bf [of Theorem \ref{th:|e'|+|e''|(i)}]}\\
{\it Step 1}.
Note that by changing the variable $y=(I+\theta)(x)$ we have
\begin{eqnarray}
 \E(\theta)
& =&
\int_{\Omega_0}
K(I+\theta,\hat{U}_\theta,{^t}[T]\cdot\nabla\hat{U}_\theta)
\det[I+\nabla\theta]dx
\nonumber\\
&=&
\int_{\Omega_0}
\Big(
\frac{1}{2}\hat{\alpha}_{kl}\partial_k\hat{U}_\theta\partial_l\hat{U}_\theta
+
\frac{1}{2}\hat{\alpha}_{00}\hat{U}^2_\theta
+
\hat{\beta}_k\partial_k\hat{U}_\theta+\hat{\gamma}\hat{U}_\theta+\hat{\delta}
\Big),								\label{e:e(th)(i),Om0}
\end{eqnarray}
where
\[
\begin{array}{rclccrcl}
\hat{\alpha}_{kl}&=&\alpha_{ij}\circ(I+\theta){T}_{ki}{T}_{lj}{\rm det}[I+\nabla\theta],
&&
\hat{\beta}_i&=&\beta_i\circ(I+\theta){T}_{ki}{\rm det}[I+\nabla\theta],\\
\hat{\gamma}&=&\gamma\circ(I+\theta){\rm det}[I+\nabla\theta],
&&
\hat{\delta}&=&\delta\circ(I+\theta){\rm det}[I+\nabla\theta].
\end{array}
\]

The differentiability of $\E$ follows from the regularity of $K$ and  i), Theorem \ref{th:hU',hU'',H12(i)}.\\
{\it Step 2}.
Assume $\Omega_0$ is Lipschitz. Differentiating (\ref{e:e(th)(i),Om0})  gives
\begin{eqnarray}
\hspace*{-7mm}
\E'(\theta)(\xi)
&=&
\int_{\Omega_0}
\hat{\alpha}_{kl}\partial_k\hat{U}'_\theta\partial_l\hat{U}_\theta
+
\hat{\alpha}_{00}\hat{U}_\theta\hat{U}'_\theta
+
\hat{\beta}_i\partial_k\hat{U}'_\theta+\hat{\gamma}\hat{U}'_\theta
\nonumber\\
&&+
\int_{\Omega_0}
\Big(
\frac{1}{2}
\hat{\alpha}_{kl}'\partial_k\hat{U}_\theta\partial_l\hat{U}_\theta
+
\frac{1}{2}
\hat{\alpha}_{00}'\hat{U}^2_\theta
+
\hat{\beta}_k'\partial_k\hat{U}_\theta
+
\hat{\gamma}'\hat{U}_\theta
+
\hat{\delta}'.								\label{e:e'(th)=(i)}
\end{eqnarray}
Taking $\theta=0$ in  (\ref{e:e'(th)=(i)}) and using (\ref{e:|hU'phi|<(i)}) gives
\begin{eqnarray}
|\E'(0)(\xi)|
&\leq&
\left|
\int_{\Omega_0}
\partial_UK(x,U_0,\nabla U_0)\hat{U}'_0
+
\partial_q K(x,U_0,\nabla U_0)\cdot\nabla\hat{U}'_0
\right|				\nonumber\\
&+&
C_1
\int_{\Omega_0}
(1+ |U_0|^2 + |\nabla U_0|^2)(|\xi|+|\nabla\xi|)
\nonumber\\
&\leq&
{
C_1
(\|\xi\|_{W^{1,q/(q-2)}(\Omega_0)} +(1+\|U_0\|^2_{W^{1,p}(\Omega_0})\|\xi\|_{W^{1,p/(p-2)}(\Omega_0)}),	\label{e:|e'(i)|<}
}
\end{eqnarray}
with $C_1=C_1(\Omega_0,L,f,K,p)$ and $p, q\in(2,\ol{p}_1)$.

Differentiating (\ref{e:e'(th)=(i)}) at $\theta=0$, isolating the terms with $\hat{U}''_0$, using the $W^{2,\infty}$-regularity of the coefficients of $L$ and $K$, the $W^{1,p}(\Omega_0)$ regularity of $U_0$  and the estimates
(\ref{e:|hU'|H1<(i)}), (\ref{e:|hU''phi|<(i)}) yields
\begin{eqnarray}
|\E''(0)(\xi,\xi)|
&\leq&
\left|
\int_{\Omega_0}
\partial_UK(x,U_0,\nabla U_0)\hat{U}''_0
+
\partial_q K(x,U_0,\nabla U_0)\cdot\nabla\hat{U}''_0
\right|				\nonumber\\
&&
+
{
C_2
\int_{\Omega_0}
\bigg[|\hat{U}'_0|^2+|\nabla\hat{U}'_0|^2 +	
(1+ |U_0|^2 + |\nabla U_0|^2)(|\xi|^2+|\nabla\xi|^2)\bigg]
}
\nonumber\\
&\leq&
C_2(
\|\xi\|^2_{W^{1,2q/(q-2)}(\Omega_0)} 
+
(1+\|U_0\|^2_{W^{1,p}(\Omega_0})
\|\xi\|^2_{W^{1,2p/(p-2)}(\Omega_0)}
\Big),								\label{e:|e''(i)|<}
\end{eqnarray}
{with $C_2=C_2(\Omega_0,L,f,K,p)$ and $p,q\in(2,\ol{p}_1)$}. 

{Taking $q=p$ in (\ref{e:|e'(i)|<}), (\ref{e:|e''(i)|<}) proves ii)}.

\noindent
{\it Step 3}.
If $\Omega_0$ is { semi}-convex then from $U_0\in W^{1,\infty}(\Omega_0)$ and  i), Proposition \ref{p:hU',hU'',H12(i),conv}, we can
take $p=\infty$ and $q\in(2,\infty)$ in (\ref{e:|e'(i)|<}) and (\ref{e:|e''(i)|<}), which proves iii).

\noindent
{\it Step 4}.
Finally, if $\Omega_0$ is { semi}-convex, $L=L^*$ and $K$ is the energy associated to $L$ then
from $U_0\in W^{1,\infty}(\Omega_0)$ and ii), Proposition \ref{p:hU',hU'',H12(i),conv}, we can take $q=p=\infty$ in (\ref{e:|e'(i)|<}) and (\ref{e:|e''(i)|<}), which proves iv).
\end{proof}

\begin{remark}\label{r:|e'|+|e''|(i),Om-conv,K-gen(i)}
{The technique we have used to obtain estimates for  $E'(\Om)$ and $E''(\Om)$, can be applied to shape functionals $E$ defined as in \eqref{e:E(i)} but involving more general $K$.
Once the differentiability of $E$ is proven, our computations can be used similarly to get estimates whose exponents will depend on the growth of $K$ at infinity.
Note that the proof of the differentiability of $E$ may rely in these cases on the differentiability of the map 
$\theta\in\Theta\mapsto \hat{U}(\theta)\in W^{1,p}(\Omega_0)$ with certain $p\geq2$, which is not known as far as we know.
}
\end{remark}

\subsection{Shape derivative estimates for an exterior PDE}\label{ss:e'',ext}
In this section we will apply the technique described in Section \ref{ss:method} to estimate the first and second 
order shape derivatives of the energy (\ref{e:E(e)}), related to the problem (\ref{e:LU=f(e)}) in the exterior of a domain $\Om$.

Even in the case when $\Omega$ is convex, its exterior $\Omega^e$ is just a Lipschitz domain.
Therefore, subject to the existence, uniqueness and regularity of the solution of the problem (\ref{e:LU=f(e)}), which represents some particularities as it is in the exterior of $\Omega$, the method we developed in Section \ref{ss:method} and used
for the interior problem in Section \ref{ss:e'',int}, also  applies for the problem in the exterior.

As in the previous section, we consider $\Theta=W^{1,\infty}(B_{R},\mathbb R^N)$ and {all the domains $\Om$ under consideration will be assumed to satisfy at least (\ref{Omega}).
}
\subsubsection{Existence, uniqueness and regularity}\label{sss:state(e)}
To analyze problem (\ref{e:LU=f(e)}), it is appropriate to consider the space $H^{1,0}(\Omega)$ as introduced in 
(\ref{e:H10}).
This space is a Banach space, see \cite{nedelec-1}, and even an Hilbert space if equipped with the inner product
\begin{equation}\label{e:H10-H}
(U,V)_{H^{1,0}(\Omega^e)}
:=
\int_{\Omega^e}
\left(
\frac{UV}{(1+|x|^2)(\ln(2+|x|^2)^{\delta_{2,N}}}
+ 
\nabla U\cdot\nabla V
\right)dx.
\end{equation}

Let  $H^{1,0}_0(\Omega^e)$ denote the closure of $\mathcal D(\Omega^e)$ in $H^{1,0}(\Omega^e)$. It can be shown that
$H^{1,0}_0(\Omega^e)$ equipped with the inner product 
\begin{equation}\label{e:H100-H}
(U,V)_{H^{1,0}_0(\Omega^e)}
:=
\int_{\Omega^e}
\nabla U\cdot\nabla Vdx,
\end{equation}
is an Hilbert space, see \cite{nedelec-1}.
Furthermore, the norms in $H^{1,0}_0(\Omega^e)$ generated by the inner products (\ref{e:H10-H}) and 
(\ref{e:H100-H}) are equivalent. Let $H^{-1,0}(\Omega^e)$ denote the dual space of $H^{1,0}_0(\Omega^e)$\\

\noindent
For the solution of (\ref{e:LU=f(e)}) we have the following regularity result.
\begin{proposition}\label{p:LU=f(e)+}
{Let $\Omega$ be as in (\ref{Omega}). Assume it is also Lipschitz and simply connected.} Let $f\in L^\infty(\mathbb R^N)$, 
${\rm supp}(f)\subset B_{R}$.
\begin{itemize}
\item[i)] 
Then there exists  a unique weak solution $U\in H^{1,0}_0(\Omega^e)$ to (\ref{e:LU=f(e)}) satisfying
\begin{equation}\label{e:|U|H10<+}
\|U\|_{H^{1,0}(\Omega^e)}\leq C(\Omega,R)\|f\|_{H^{-1,0}(\Omega^e)}.
\end{equation}
\item[ii)]
There exists $p_{1}=p_1(\Om_{0})$ satisfying [$p_1>4$ if $N=2$ and $p_1>3$ if $N\geq 3$], such that $U\in W^{1,p}_0(\Omega^e)$ for any $p\in(p_{1}',p_1)$ where $1/p_1+1/p_1'=1$, and
\begin{equation}\label{e:|U|W1ploc<+}
\|U\|_{W^{1,p}(\Omega^e_R)}\leq C(\Omega,R,p)\|f\|_{L^\infty},\quad \Omega^e_R=\Omega^e\cap B_{R}.
\end{equation}
\end{itemize}\end{proposition}
{\bf Proof}.\\
{\it Step 1.}
Note that the weak solution $U$ of  (\ref{e:LU=f(e)}) satisfies
\begin{equation}\label{e:LU=f(e,w)}
\int_{\Omega^e}\nabla U\cdot\nabla\varphi = \int_{\Omega^e}f\varphi,\quad\forall \varphi\in H^{1,0}_0(\Omega^e).
\end{equation}
Note also that as $f\in L^\infty(\mathbb R^N)$ has compact support we then have $f\in H^{-1,0}(\Omega^e)$.
Therefore we have existence and uniqueness of a solution $U\in H^{1,0}_0(\Omega^e)$ to (\ref{e:LU=f(e,w)}), and  estimate 
(\ref{e:|U|H10<+}) follows directly from  Lax-Milgram Lemma and from (\ref{e:H100-H}).
\\
{\it Step 2.}
We will use \cite[Theorem 0.5]{JK-1}. 
Let $\eta\in\mathcal D(\mathbb R^N)$, $\eta\in[0,1]$, $\eta=1$ in $B_{R}$ and
$\eta=0$ in $\mathbb R^N\backslash B_{2R}$. Then $\eta U$ satisfies
\begin{eqnarray*}
 -\Delta(\eta U)
 &=&
 f\eta -2\nabla U\cdot\nabla\eta - U\Delta\eta =: g\;\, \textrm{ in }\; \Omega^e_{2R}:=\Omega^e\cap B_{2R},
 \quad \\
 \eta U &=&0\;\, \textrm{ on }\; \partial\Omega^e_{2R}.
\end{eqnarray*}
From i) above $g\in L^2(\Omega^e_{2R})$. By local regularity of $-\Delta$, as $f\in L^\infty(\mathbb R^N)$ we get
$U\in W^{1,p}(B_{R,2R})$ for all $p\in(1,\infty)$, where  $B_{R,2R}:=\{x,\, R<|x|<2R\}$. 
Therefore, as $\eta=1$ in $\Omega^e_R$ it follows $\nabla\eta=0$ in $\Omega^e_R$, so we get $g\in W^{-1,p}(\Omega^e_{2R})$ and
\[
 \|g\|_{W^{-1,p}(\Omega^e_{2R})}
 \leq
 C(\Omega,R,p)
 \|f\|_{L^\infty}.
\]
We recall \cite[Theorem 0.5]{JK-1}, which states that there exists $p_1$ depending on $\Omega^e_{2R}$, 
satisfying $p_1>3$ for every $N\geq 3$ and $p_1>4$ if $N=2$, such that
the map $U\in W^{1,p}_0(\Omega^e_{2R})\mapsto -\Delta U\in W^{-1,p}(\Omega^e_{2R})$ defines an isomorphism for every $p\in(p_{1}',p_{1})$. It implies that $\|\eta U\|_{W^{1,p}(\Omega^e_{2R})}\leq C(\Omega,R,p)\|f\|_{L^\infty}$, which implies (\ref{e:|U|W1ploc<+}).
\hfill$\Box$

{
\begin{remark}\label{r:p1*,L*(e)}
In connection with Remark \ref{r:p1*,L*} related to the interior problem, note that with our choice 
of $L=L^*=-\Delta$, we have here $p_1=p^*_1=\ol{p}_1$. Recall that $\ol{p}_1>4$ if $N=2$ and $\ol{p}_1>3$ if $N\geq 3$.
\end{remark}
}

\subsubsection{Estimates of the shape derivatives of the solution}\label{sss:hU',hU''(e)}
We have the following theorem, which is similar to Theorem \ref{th:hU',hU'',H12(i)}:
\begin{theorem}\label{th:hU',hU'',H12(e)}
{Let $\Omega$ be as in (\ref{Omega}). Assume it is also Lipschitz and simply connected.} Let $\theta\in\Theta$, $\Omega^e_\theta=(I+\theta)(\Omega^e_0)$,
$f\in L^\infty(\mathbb R^N)$ with ${\rm supp}(f)\subset B_{R}$, 
$U_\theta\in H^{1,0}_0(\Omega^e_\theta)$ the solution of (\ref{e:LU=f(e)}) (see  i) in Proposition \ref{p:LU=f(e)+}), 
$\hat{U}_\theta=U_\theta\circ(I+\theta)$.
Then
\begin{itemize}
\item[i)]
If $f\in W^{m,\infty}(\Omega^e_{2R})$, $m\in\mathbb N^*$, then $\theta\mapsto\hat{U}_\theta\in H^1_0(\Omega_0)$ is of class $C^m$ in a neighborhood of $\theta=0\in\Theta$.
\end{itemize}
{\it In the following, we assume the previous hypotheses are satisfied {for $m=2$. }}
\begin{itemize}
\item[ii)]
Furthermore, if $\Omega_0$ is Lipschitz, then for all $p\in(2,\ol{p}_1)$ where $\ol{p}_{1}$ is introduced in  Remark \ref{r:p1*,L*(e)}, we have
\begin{eqnarray}
\|\hat{U}'_0\|_{H^{1,0}(\Omega^e_0)}
&\leq&
 C _1
 \|\xi\|_{W^{1,2p/(p-2)}(\Omega^e_R)},			\label{e:|hU'|H1<(e)}\\
\left|
\int_{\Omega^e_0}
\nabla U_0\cdot\hat{U}'_0
\right|
&\leq&
 C_1 \|\xi\|_{W^{1,p/(p-2)}(\Omega^e_R)},		\label{e:|hU'phi|<(e)}\\
\left|
\int_{\Omega^e_0}
\nabla U_0\cdot\hat{U}''_0
\right|
&\leq&
C_2
\|\xi\|^2_{W^{1,2p/(p-2)}(\Omega^e_R)},		\label{e:|hU''phi|<(e)}
\end{eqnarray}
where
$\hat{U}'_0=\hat{U}_\theta'(0)(\xi)$, 
$\hat{U}''_0=\hat{U}_\theta''(0)(\xi,\xi)$,
$C_i= C_i(\Omega_0,U_0,f,R,p)$.
\end{itemize}
\end{theorem}
{\bf Proof}.\\
{\it Step 1}.
The proof of i) is similar to the one of the same result in the interior, see i) in Theorem \ref{th:hU',hU'',H12(i)}.
For convenience we present the proof.
 
The proof is based on the implicit function theorem.
After changing the variable $x=(I+\theta)(y)$, the weak form (\ref{e:LU=f(e,w)}) in $\Omega^e_\theta$ is transformed in 
$\Omega^e_0$ as follows
\begin{equation}
 \int_{\Omega^e_0}
 \nabla\hat{U}_\theta \cdot M_\theta \cdot\nabla\varphi - \hat{f}_\theta\varphi  = 0, 				\label{e:LhU=hf,w(e)}
\end{equation}
where  $M_\theta=(T\cdot {^tT}){\rm det}[Id+\nabla\theta]$,  ${T}=[Id+\nabla\theta]^{-1}$ and $\hat{f}=f\circ(I+\xi(t)){\rm det}[I+\nabla\theta]$.
Consider $F$ defined by
\[
 \begin{array}{lclll}
 F:&\Theta\times H^{1,0}_0(\Omega^e_0)&\mapsto&H^{-1,0}(\Omega_0)\\
    &(\theta,\hat{U})&=&
\left[\varphi\mapsto \int_{\Omega_0}
 \nabla\hat{U} \cdot M_\theta \cdot\nabla\varphi - \hat{f}_\theta\varphi\right].
 \end{array}
\]
It is easy to check that $F$ is well defined and of class $C^m$ in a neighborhood of $(0,U_0)\in \Theta\times H^{1,0}_0(\Omega^e_0)$ 
(here we use the fact that $f$ is with compact support), and 
\[
 \partial_{{U}}F(0,U_0)(U):\varphi
 \mapsto
 \int_{\Omega_0}
 {(\nabla U\cdot\nabla \varphi)}dx,
 \quad\forall U, \varphi\in H^{1,0}_0(\Omega^0_0).
\]
Note that $\partial_{{U}}F(0,U_0)$ defines an isomorphism from $H^{1,0}_0(\Omega^e_0)$ to $H^{-1,0}(\Omega^e_0)$,
see Proposition \ref{p:LU=f(e)+}.
Then, from implicit function Theorem, there exists a $C^m$ map $\theta\mapsto{U}(\theta)$ such that $F(\theta,{U}(\theta)) =0$ 
for $\|\theta\|_\Theta$ small.
From the uniqueness of solution to (\ref{e:LU=f(e)}) we obtain $\hat{U}_\theta={U}(\theta)$, which proves the differentiability of $\theta\mapsto\hat{U}_\theta$.

\noindent
{\it Step 2}.
Again the proofs of  (\ref{e:|hU'|H1<(e)})-(\ref{e:|hU''phi|<(e)}) are similar to the proofs of 
(\ref{e:|hU'|H1<(i)})-(\ref{e:|hU''phi|<(i)}) (except that the choice of suitable test functions is easier here). For the proof of (\ref{e:|hU'|H1<(e)}) we differentiate (\ref{e:LhU=hf,w(e)}) and get
\begin{eqnarray}
\hspace*{-6mm}
\int_{\Omega^e_0}
(\nabla\hat{U}_\theta'\cdot M_\theta \cdot\nabla\varphi)
&=&
\int_{\Omega^e_0}
\hat{f}'_\theta \varphi - \nabla\hat{U}_\theta\cdot M'_\theta\cdot \nabla\varphi.			\label{e:(LU-f)'=0(e)}
\end{eqnarray}
With $\theta=0$ and $\varphi=\hat{U}'_0$, as ${\rm Supp}(\xi)\subset B(0,R)$, from (\ref{e:(LU-f)'=0(e)}) we get
\begin{eqnarray}
\hspace*{-6mm}
\|\nabla\hat{U}'_0\|_{L^2(\Omega^e_0)}^2
&\leq&
C_1
\int_{\Omega^e_R}
(1+|\nabla U_0|)(|\xi|+|\nabla\xi|)(|\hat{U}'_0|+|\nabla\hat{U}'_0|)			\nonumber\\
&\leq&
C_1
(1+\|U_0\|_{W^{1,p}(\Omega^e_R)})
\|\xi\|_{W^{1,2p/(p-2)}(\Omega^e_R)}
\|\nabla \hat{U}'_0\|_{L^2(\Omega^e_0)},						\label{e:hU',H12(e-1)}
\end{eqnarray}
with $C_1=C_1(\Omega_0,R,f,p)$, which proves  (\ref{e:|hU'|H1<(e)}).

Similarly, if we take $\theta=0$ and $\varphi=U_0$ in (\ref{e:(LU-f)'=0(e)}) we get
\begin{eqnarray*}
\left|
\int_{\Omega^e_0}
\nabla\hat{U}'_0\cdot\nabla U_0
\right|
&\leq&
C_1
\int_{\Omega^e_R}
(1+|\nabla U_0|)(|\xi|+|\nabla\xi|)|\nabla U_0|				\nonumber\\
&\leq&
C_1(1+\|U_0\|^2_{W^{1,p}(\Omega^e_R)})\|\xi\|_{W^{1,p/(p-2)}(\Omega^e_R)},		
\end{eqnarray*}
with $C_1=C_1(\Omega_0,f,R,p)$, which proves (\ref{e:|hU'phi|<(e)}).

\noindent
{\it Step 3}.
For the proof of (\ref{e:|hU''phi|<(e)}) we differentiate (\ref{e:(LU-f)'=0(e)}) at $\theta=0$ and take ${\varphi}=U_0$.
It gives
\begin{eqnarray*}
\left|
\int_{\Omega^e_0}
\nabla\hat{U}''_0\cdot\nabla U_0
\right|
&\leq&
C_2
\int_{\Omega^e_R}
(1+|\nabla{U}_0|^2)(|\xi|^2+|\nabla\xi|^2)	
+
|\nabla{U}_0||\nabla\xi||\nabla\hat{U}'_0|		\nonumber\\
&\leq&
C_2
\Big(
\|U_0\|_{W^{1,p}(\Omega^e_R)}\|\nabla\xi\|_{L^{2p/(p-2)}(\Omega^e_R)}\|\nabla\hat{U}_0'\|_{L^2(\Omega^e_R)}
+ 
\nonumber\\
&&
\hspace*{6mm}
(1+\|U_0\|^2_{W^{1,p}(\Omega^e_R)})\|\xi\|^2_{W^{1,2p/(p-2)}(\Omega^e_R)}
\Big)	
\nonumber\\
&\leq&
C_2(1+\|U_0\|^2_{W^{1,p}(\Omega^e_R)}\|\xi\|^2_{W^{1,2p/(p-2)}(\Omega^e_R)},		
\end{eqnarray*}
with $C_2=C_2(\Omega_0,f,R,p)$, which completes the proof.
\hfill$\Box$

\subsubsection{Estimate of shape derivatives of the energy}
\begin{theorem}\label{th:e',e''(e)}
{Let $\Omega$ be as in (\ref{Omega}). Assume it is also Lipschitz and simply connected.} Let 
$\theta\in\Theta$, $\Omega^e_\theta=(I+\theta)(\Omega^e_0)$,
$f\in W^{m,\infty}(\mathbb R^N)$, $m\in\N^*$, with  ${\rm supp}(f)\Subset B_{R}$, $U_\theta\in H^{1,0}_0(\Omega^e_\theta)$ the solution of (\ref{e:LU=f(e,w)}) and $\E(\theta)=\int_{\Omega_\theta}|\nabla U_\theta|^2$. 
Then $\E$ is of class $C^m$ near $\theta=0\in\Theta$ and for $m\geq 2$ and for all $p\in (2,\ol{p}_1)$ with $\ol{p}_1$ defined in Remark \ref{r:p1*,L*(e)}, we have
\begin{eqnarray}
 |\E'(0)(\xi)|
 &\leq& 
 C_1
\|\xi\|_{W^{1,p/(p-2)}(\Omega^e_R)},			\label{e:|e'(e)|<,Lip}
\\
 |\E''(0)(\xi,\xi)|
 &\leq& 
 C_2
\|\xi\|^2_{W^{1,2p/(p-2)}(\Omega^e_R)},				\label{e:|e''(e)|<,Lip}
\end{eqnarray}
where $C_i=C_i(\Omega_0,U_0,f,R,p)$.
\end{theorem}
{\bf Proof}.
The proof of the theorem is very similar to the one of Theorem \ref{th:hU',hU'',H12(i)}. We present it briefly.
Note that
\begin{eqnarray}
\E(\theta)
&=&
\int_{\Omega^e_0}
\nabla\hat{U}_\theta\cdot M_\theta\cdot\nabla\hat{U}_\theta.			\label{e:e(th)(e)}
\end{eqnarray}
From Theorem \ref{th:hU',hU'',H12(e)} it follows that $\E$ is $C^m$ in a neighborhood of $\theta=0$. 

Differentiating (\ref{e:e(th)(e)}) at $\theta=0$ in a direction $\xi\in\Theta$, and then using (\ref{e:|hU'phi|<(e)}) gives
\begin{eqnarray}
|\E'(0)(\xi)|
&\leq&
C_1
\left(
\left|\int_{\Omega^e_0}\nabla U_0\cdot\nabla \hat{U}_0'
\right|
+
\int_{\Omega^e_R}
|\nabla U_0|^2|\nabla\xi|	
\right)															\nonumber\\
&\leq&
C_1
\|\xi\|_{W^{1,p/(p-2)}(\Omega^e_R)},												\label{e:|e'(e)|<}
\end{eqnarray}
with $C_1=C_1(\Omega_0,U_0,R,p,f)$, which  proves (\ref{e:|e'(e)|<,Lip}).

Differentiating  twice (\ref{e:e(th)(e)}) at $\theta=0$ in the direction $\xi$, and then using 
(\ref{e:|hU'|H1<(e)}) and (\ref{e:|hU''phi|<(e)}) gives 
\begin{eqnarray}
|\E''(0)(\xi,\xi)|
&\leq&
C_2
\left|
\int_{\Omega^e_0}\nabla\hat{U}''_0\cdot\nabla U_0
\right|
+
\int_{\Omega^e_0}
|\nabla\hat{U}'_0|^2
+
|\nabla U_0|^2(|\xi|^2+|\nabla\xi|^2)			\nonumber\\
&\leq&
C_2
\|\xi\|^2_{W^{1,2p/(p-2)}(\Omega^e_0)},
\end{eqnarray}
with $C_2=C_2(\Omega_0,U_0,R,p,f)$, which completes the proof.
\hfill$\Box$

\section{Application to optimal convex shapes}\label{sect:app}

In this section, we remind the strategy from \cite{LNP}, and emphasize an application of the second order shape derivatives estimates obtained in Section \ref{ss:mainresults}, first in dimension two, then in higher dimensions.

\subsection{Application in the planar case}\label{ssect:planar}

Here is the main result of this subsection:
\begin{theorem}\label{th:polygon}
Let $\Omega_0$ an open convex subset of $\R^2$, solution of the optimization problem:
\begin{equation}\label{e:minJbis}
\hspace*{-8mm}
\min\Big\{J(\Omega)=R(E(\Omega),|\Omega|)-P(\Omega),\;\; \Omega\subset\R^2\textrm{open, convex and such that }\partial\Om\subset\{x, a\leq |x|\leq b\}\Big\},
\end{equation}
where $R:\R^2\to\R$ is smooth, $(a,b)\in(0,\infty]^2$, $a<b$, and $E$ is an energy like (\ref{e:E(i)}), corresponding to the interior problem,  or like (\ref{e:E(e)}), corresponding to the exterior problem.

Then every connected component of $(\partial\Om_{0})_{in}:=\partial\Om_{0}\cap\{x, a< |x|< b\}$ is a finite union of straight segments.
\end{theorem}
We insist here on the fact that the existence of an optimal shape for problem (\ref{e:minJbis}) is true and easy to obtain in the case $0<a<b<\infty$, see Remark \ref{rk:existence}. The cases $b=\infty$ and/or $a=0$ require more attention: it may happens that minimizing sequences are not bounded (for example if $R=0$), or that they converge to a segment.

\begin{remark}\label{rk:dir}
This result was obtained in \cite{LNP} for the two following particular cases:
\begin{itemize}
\item $E=E_{f}$ is the Dirichlet energy associated to $f\in H^2_{loc}(\R^2)$:
 $$E(\Om)=\int_{\Om}\left(\frac{1}{2}|\nabla U_{\Om}|^2-fU_{\Om}\right), \;\;\textrm{where}\;\;U_{\Om}\;\;\textrm{solves}\;\;\left\{\begin{array}{l}
 -\Delta U_{\Om}=f\;\textrm{in}\;\Om\\
 U_{\Om}\in H^1_{0}(\Om)
 \end{array}\right.$$
\item $E=\lambda_{1}$ is the first Dirichlet eigenvalue of $\Om$:
$$\lambda_{1}(\Om)=\min\left\{\frac{\int_{\Om}|\nabla U|^2}{\int_{\Om} U^2}, \;\;U\in H^1_{0}(\Om)\setminus\{0\}\right\}.$$
\end{itemize}
\end{remark}

{\begin{remark} In Theorem \ref{th:polygon},  we chose a simple function of $P(\Omega)$. Actually, the same result would extend to functionals of the form $R\big(E(\Omega), |\Omega|,\lambda_1(\Omega), P(\Omega)\big)$ with $p\to R(\cdot,\cdot,\cdot,p)$ being decreasing and concave.
\end{remark}
 
\begin{remark}\label{rk:constraint}
Similarly to \cite[Theorem 4, Example 8]{LNP}, Theorem \ref{th:polygon} is also valid with a volume constraint. More precisely, any solution of
\begin{equation}
\min\{J(\Omega)=R(E(\Omega))-P(\Omega),\;\; \Omega\subset\R^2\textrm{ open, convex and such that }|\Om|=V_{0}\}
\end{equation}
(where $R:\R\to\R$ is smooth and $E$ is as in Theorem \ref{th:polygon}) is a polygon. {See also Remark \ref{rk:constraint2} for other geometrical constraints.}
\end{remark}

In order to analyze problem \eqref{e:minJbis}, we use, as in \cite{LN, LNP},  the following classical parametrization of 2-dimensional convex with polar coordinates 
$(r,\theta)\in[0,\infty)\times\T$, where $\T=\R/2\pi\Z$:
\begin{equation}\label{e:Omega_u}
\Om_u:=\left\{(r,\theta)\in [0,\infty)\times\R\;;\;r<\frac{1}{u(\theta)}\right\},
\end{equation}
where $u$ is a positive  and $2\pi$-periodic function.
A simple computation shows that the curvature of $\Om_u$ is 
\begin{equation}\label{eq:curvature}
\kappa_{\partial\Om_{u}}=\frac{u''+u}{\left(1+\left(\frac{u'}{u}\right)^2\right)^{3/2}}.
\end{equation}
 This implies that $\Om_{u}$ is convex if and only if $u''+u\geq0$, which has to be understood in the sense 
 {of $H^{-1}(\T)$} if $u$ is not $C^2$. More precisely, if {$u\in H^1(\T)$} then $u''+u\geq 0$ if and only if
\begin{equation*}
\forall\; {v\in H^1(\T)}\;\; \mbox{ with }v\geq0,\;\;\int_{\T} \left(uv - u'v'\right)d\theta\geq0.
\end{equation*}
Throughout this section, any function defined on $\T$ is considered as the restriction to $\T$ of a $2\pi$-periodic function on $\R$, with the same regularity.

With this parametrization, considering $j(u)=J(\Om_{u})$, Problem \eqref{e:minJbis} is equivalent to
\begin{equation}\label{e:minJu}
\min\left\{j(u),\; u''+u\geq0,\; u\in\mathcal U_{ad}\right\},\quad
\;\textrm{where}\;\; \mathcal U_{ad}=\{u\in W^{1,\infty}(\T),\;\, 1/u\in[a,b]\}.
\end{equation}

Then we have the following  result proven in \cite[Theorem 3]{LNP}.
\begin{theorem}\label{th:Uad->pol}
Let $u_{0}>0$ be a solution for \eqref{e:minJu} and
${\displaystyle \T_{in}:=\left\{\theta\in\T, {a<\frac{1}{u_{0}(\theta)}<b}\right\}}$.
Assume $j:W^{1,\infty}(\T)\to\R$ is $C^2$ and that  there exist  $s\in[0,1)$, $\alpha>0$, $\beta,\gamma\in\R$ such that, for any
$v\in W^{1,\infty}(\T)$,
 we have
\begin{eqnarray}
 j''(u_{0})(v,v)
 \leq
 -\alpha|v|^2_{H^1(\mathbb T)} + \gamma |v|_{H^1(\mathbb T)}\|v\|_{H^{s}(\mathbb T)}+\beta\|v\|^2_{H^{s}(\mathbb T)}.
 \label{e:coercivite}
\end{eqnarray}
If $I$ is a connected component of $\T_{in}$, then
$$
u_{0}''+u_{0}\textrm{ is a finite sum of Dirac masses in }I.
$$

\end{theorem}
This result, combined with the estimates from Section \ref{ss:mainresults} will lead to a proof of Theorem \ref{th:polygon}. Indeed, formula \eqref{eq:curvature} explains that $\Om_{u}$ is polygonal if and only if $u''+u$ is a sum of Dirac masses.

\begin{remark}\label{rk:constraint2}
As in \cite{LNP}, we may focus on a geometrical constraint different from $\partial\Om\subset\{x, a\leq |x| \leq b\}$. In that case, the previous results remain valid when replacing the definitions of $\T_{in}, (\partial\Om_{0})_{in}$ with:
\begin{equation*}
\T_{in}
=\{\theta\in\T\; /\; \exists \eps,\delta>0,  \forall v\in W^{1,\infty}_{0}(\theta-\eps,\theta+\eps)\textrm{ with }
\|v\|_{W^{1,\infty}}<\delta, u_0+v\in \mathcal{U}_{ad}\}.  
\end{equation*}
\begin{equation*}
(\partial\Omega_0)_{in}
=\left\{x\in\partial\Omega_0 \;/\; \exists \theta\in \T_{in}, x=\frac{1}{u_{0}(\theta)}(\cos\theta,\sin\theta)\right\}.
\end{equation*}
{where $\mathcal{U}_{ad}$ replaces $\{\Om/\partial\Om\subset\{x, a\leq |x| \leq b\}$ and includes the geometrical constraint, except the convexity one.}
Note that $(\partial\Omega_0)_{in}$ describes the part of the boundary which is {\it inside} the constraints, apart from the convexity one. However, if one deals also with a finite-dimensional equality constraint (for example $|\Om|=V_{0}$), then one should work with a Lagrange multiplier, see Remark \ref{rk:constraint}.
\end{remark}

\begin{remark}\label{rk:existence}
Due to the convexity constraint, the existence of a solution to (\ref{e:minJbis})  can be proved easily when $0<a<b<\infty$.
We briefly provide a proof here: let $(\Omega_n)_{n\in\N}\in\mathcal{O}_{ad}^\N$, with parametrization $u_n\in\mathcal U_{ad}$, be a minimizing sequence of 
$J(\Omega)$.
From $u_n''+u_n\geq0$ and $1/u_n\in[a,b]$, it follows that $(u_n)_{n\in\mathbb N}$, is strongly relatively compact in {$H^1(\T)$}
and therefore we may assume $\lim_{n\to\infty}u_n=u_0$ in ${H^1}(\T)$ (see also Remark \ref{rk:existmultiD}{)}, with $u_0''+u_0\geq0$ and $1/u_0\in[a,b]$. 
In such conditions the domains $\Omega_n$ satisfy a uniform exterior cone condition and will converge to $\Omega_0$ {for the Hausdorff convergence},
with parametrization $u_0$. Then $U_{\Omega_n}$ will converge to $U_{\Omega_0}$ in $H^1(B(0,b))$ 
(see \cite[Prop 2.4.4 and Theorem 3.2.13]{HP}), {and thanks to the continuity of the perimeter for the $H^1(\T)$-norm we get that $P(\Om_{n})$ converges to $P(\Om_{0})$}, which implies that $\Omega_0$ solves (\ref{e:minJbis}).
\end{remark}
{\bf Proof of Theorem \ref{th:polygon}}.
Let $u_{0}$ such that $\Om_{0}=\Om_{u_{0}}$. Then $u_{0}$ is a solution of \eqref{e:minJu}, where 
$j(u)={R}(e(u),m(u))-p(u)$ and $e(u)=E(\Om_{u}), m(u)=|\Om_{u}|, p(u)=P(\Om_{u})$.

The second order derivatives related to the geometrical terms $m(u)$ and $p(u)$ can be easily computed. Indeed, 
first note that for every $u\in W^{1,\infty}(\T)$, $u>0$ we have
\[
m(u)=\int_{\T} \frac{1}{2u^2}d\theta, \;\; \qquad 
p(u)=\int_{\T} \frac{\sqrt{u^2+u'^2}}{u^2}d\theta.
\]
It follows that $m$ and $p$ are twice differentiable {in $W^{1,\infty}(\T)$} and
there exist some real numbers $\beta_{1},\beta_{2}, \gamma$ (depending on $u_0$) and $\alpha>0$ such that, 
$$\left\{
\begin{array}{l}
|m''(u_{0})(v,v)|\leq \beta_{1}\|v\|^2_{L^2(\T)}\\
[3mm]
p''(u_{0})(v,v)\geq \alpha|v|^2_{H^{1}(\T)}-\gamma|v|_{H^{1}(\T)}\|v\|_{L^2(\T)}-\beta_{2}\|v\|_{L^2(\T)}^2
\end{array}
\right.
$$
It remains the difficult term $e''(u_{0})(v,v)$, which relies on Section  \ref{ss:mainresults}. 
To handle this term, we consider $v\in W^{1,\infty}(\T)$ and introduce the vector field:
\begin{eqnarray}\label{e:xi(t)}
 &&\xi\in  C^2((-t_{0},t_{0}),W^{1,\infty}(\R^2,\R^2)), 
 \nonumber\\
 &&\xi(t)=\left(\frac{1}{u+tv}-\frac{1}{u}\right)e^{i\theta}\eta(r,\theta)\quad \mbox{on }\ \partial\Omega_0\;\;\textrm{(in polar coordinates),}
\end{eqnarray}
{where $\eta\in C^\infty_{c}(\R^2)=\mathcal{D}(\R^2), 0\leq \eta\leq 1$, $\eta=1$ (resp. $\eta=0$) in a neighborhood of $\partial\Omega_0$ (resp. the origin)}.
Let us point out that $(I+\xi(t))(\Om_{u})=\Om_{u+tv}$, $e(u+tv)=E(\Omega_{u+tv})=\E(\xi(t))$ and also that
\begin{equation}\label{e:xi'(u),xi''(u)}
\hspace*{-2mm}
\forall v\in W^{1,\infty}(\T):\;\;
\xi'(0)(v)
=
-\frac{v}{u^2}e^{i\theta},
\;\,
\xi''(0)(v,v)
=
2\frac{v^2}{u^3}e^{i\theta}\;\,
\textrm{on}\; \partial\Om.
\end{equation}
Then, we will differentiate twice around 0 the map $t\mapsto \varphi(t):=E(\Omega_{u+tv})=\E_{\Om}(\xi(t))$:
\begin{eqnarray*}
\varphi'(0)
&=e'(u_{0})(v)&
=
\mathcal E'(0)(\xi'(0)),		\\[1mm]
\varphi''(0)
&=e''(u_{0})(v,v)&
=
\mathcal E''(0)(\xi'(0),\xi'(0))
+
\mathcal E'(0)(\xi''(0)).		
\end{eqnarray*}
and therefore Corollary \ref{c:e',e''(i,e,intro)} (see also \eqref{eq:lip1}, \eqref{eq:lip2})  implies {there exists $\eps\in(0,1/2)$ such that}
\[
|e''(u_{0})(v,v)|\leq \beta_{3}(\eps)\|\xi''(0)\|_{H^{\frac{1}{2}-\eps}(\partial\Om)}+\beta_{4}(\eps)\|\xi'(0)\|_{H^{1-\eps}(\partial\Om)}^2.
\]
We then use  the fact that the $H^{1-\eps}(\partial\Om)$ and $H^{1-\eps}(\T)$ norms are equivalent (since the transformation 
$\psi=\psi(r,\theta):=\frac{r}{u(\theta)}e^{i\theta}$ is bi-Lipschitz near $\T$ and $\psi(\T)=\partial\Omega$), 
and the fact that $H^{1-\eps}(\T)$ is a Banach algebra
to obtain
$$|e''(u)(v,v)|\leq \beta_{5}(\eps)\|v\|_{H^{1-\eps}(\T)}^2,$$
Easy computations about the term ${R}(E(\Om),|\Om|)$ imply that $j''$ satisfies \eqref{e:coercivite},
and therefore Theorem \ref{th:Uad->pol} applies. It follows that $u_{0}''+u_{0}$ is a finite sum of Dirac masses in any connected component of $\T_{in}$. Geometrically speaking, considering the formula \eqref{eq:curvature}, this correspond to the fact that any connected component of $(\partial\Om_{0})_{in}$ is polygonal and concludes the proof.
\hfill$\Box$

{
\subsection{Application in the multi-dimensional case}

In the multi-dimensional case, convexity constraint in shape optimization is much less understood, see \cite{BFL,HHL} and the work of T. Lachand-Robert. Nevertheless, we can use a parametrization similar to the one used in Section \ref{ssect:planar}, and we show in this section that our estimates of shape derivatives allows to obtain results in any dimension.

For $N\geq 2$, if $u:\S^{N-1}\to(0,\infty)$ is given, $\S^{N-1}=\{x\in\R^N,\; |x|=1\}$, we can consider
\begin{equation}\label{e:Omega_ubis}
\Om_u:=\left\{(r,\theta)\in [0,\infty)\times\S^{N-1},\;\;\;r<\frac{1}{u(\theta)}\right\}.
\end{equation}
{The function $u$ is the so-called gauge function of $\Om_{u}$.} 
The set $\Om_u$ is convex if and only if the 1-homogeneous extension of $u$, {denoted by the same letter} and given by 
{${\displaystyle u(x)=|x|u(x/|x|)}$} is convex in $\R^N$ (in this section, we will refer to this property by saying that $u:\S^{N-1}\to\R$ is convex), {see \cite[Section 1.7]{S} for example}.  In this way, we describe every bounded convex open set containing the origin.
Throughout this section, the regularity of any function defined on $\S^{N-1}$ is seen as the regularity on $\R^N\setminus\{0\}$ of its 1-homogeneous extension, and it is classical that it is equivalent to the regularity of the set $\Om_{u}$ itself.

With this parametrization, considering $j(u)=J(\Om_{u})$, problem \eqref{e:minJ} is equivalent to
\begin{equation}\label{e:minJubis}
\min\left\{j(u),\; u:\S^{N-1}\to(0,\infty)\textrm{ convex },\; u\in\mathcal U_{ad}\right\},
\end{equation}
{where $\mathcal U_{ad}\subset W^{1,\infty}(\S^{N-1})$ is defined as $\mathcal{U}_{ad}=\{u\in W^{1,\infty}(\S^{N-1}),\;\, \Om_{u}\in\mathcal{O}_{ad}\}$ and $\mathcal O_{ad}$ is the class of admissible open bounded sets considered in \eqref{e:minJ} (not taking into account the convexity constraint).}

Then in the same spirit as Theorem \ref{th:Uad->pol}, we can prove:
\begin{theorem}\label{th:multiD}
Let $u_{0}>0$ be a solution for \eqref{e:minJubis}, where $\mathcal{U}_{ad}$ is a convex subset of $W^{1,\infty}(\S^{N-1})$.
Assume $j:W^{1,\infty}(\S^{N-1})\to\R$ is $C^2$ and that  there exists  $s\in[0,1)$, $\alpha>0$, $\beta,\gamma\in\R$ such that for any
$v\in W^{1,\infty}(\S^{N-1})$
 we have
\begin{eqnarray}\label{eq:conc}
 j''(u_{0})(v,v)
 \leq
 -\alpha|v|^2_{H^1(\S^{N-1})} + \gamma |v|_{H^1(\S^{N-1})}\|v\|_{H^{s}(\S^{N-1})}+\beta\|v\|^2_{H^{s}(\S^{N-1})}.
 \label{e:coercivitebis}
\end{eqnarray}
Then the set
\begin{equation}\label{eq:tangent}
T_{u_{0}}=\{ v\in W^{1,\infty}(\S^{N-1}) / \exists \eps>0, \forall |t|<\eps,\; u_{0}+tv\in{\mathcal U}_{ad} \textrm{ and is convex}\},
\end{equation}
is a linear vector space of finite dimension.
\end{theorem}

\begin{proof}
It is easy to check that $T_{u_{0}}$ is a linear vector space. Since $u_{0}$ is optimal, we have $j''(u_{0})(v,v)\geq 0$ for every $v\in T_{u_{0}}$. To conclude, we prove that the unit ball of $T_{u_{0}}$ for the {$H^s(\S^{N-1})$}-norm is relatively compact. Indeed, thanks to \eqref{eq:conc}, we get
$$\forall v\in T_{u_{0}}\textrm{ such that }\|v\|_{H^{s}(\S^{N-1})}\leq 1, \;\;\alpha|v|^2_{H^1(\S^{N-1})}\leq \gamma |v|_{H^1(\S^{N-1})}+\beta,$$
which easily leads to
$$\forall v\in T_{u_{0}}\textrm{ such that }\|v\|_{H^{s}(\S^{N-1})}\leq 1, \;\;
|v|_{H^1(\S^{N-1})}\leq {\frac{|\gamma|+\sqrt{\gamma^2+4\alpha|\beta|}}{2\alpha}}.$$
From the compact embedding of $H^{s}(\S^{N-1})$ in $H^{1}(\S^{N-1})$, we conclude that the unit ball of $T_{u_{0}}\subset H^s(S^{\N-1})$ has a compact closure and therefore from Riesz Theorem, $T_{u_{0}}$ has finite dimension.
\end{proof}

As a corollary, the results we obtained in this paper, {namely iii), Theorem \ref{th:e',e''(i,intro)}}, lead to the following generalization and improvement of \cite[Theorem 4.5]{BFL}:
\begin{corollary}\label{cor:multiD}
Let $\Omega_0$ be an open convex subset of $\R^{N}$, solution of the optimization problem:
\begin{equation}\label{e:minJbisN}
\hspace*{-8mm}
\min\Big\{J(\Omega)=R(E(\Omega),|\Omega|)-P(\Omega),\;\; \Omega\subset\R^N\textrm{open, convex and such that }\partial\Om\subset\{x, a\leq |x|\leq b\}\Big\},
\end{equation}
where $R:\R^2\to\R$ is smooth, $(a,b)\in(0,\infty]^2$, $a<b$, and $E$ is an energy like (\ref{e:E(i)}), corresponding to the interior problem.

Then denoting by $u_{0}$ the gauge function of $\Om_{0}$, $T_{u_{0}}$ defined in \eqref{eq:tangent} has finite dimension. In particular, if $\omega$ is a $C^2$ relatively open subset of $(\partial\Om_{0})_{in}:=\partial\Om_{0}\cap\{x, a< |x|< b\}$, then the Gauss curvature of $\Om_{0}$ vanishes on $\omega$.
\end{corollary}
Compared to \cite[Theorem 4.5]{BFL}, we have enlarged the class of functionals under consideration, and we require less regularity on the optimal set.
\begin{remark}
Contrary to Theorem \ref{th:polygon}, we cannot deal with the exterior problem when $N>2$. This is due to the fact that for a Lipschitz domain $\Om$ in dimension $N\geq 3$, our estimates do not imply that there exists $s<1$ such that $E''(\Om)(\xi,\xi)\leq \|\xi\|_{H^s(\partial\Om)}$ (whereas it is the case for $N=2$, see \eqref{eq:lip2}).
\end{remark}

\begin{remark}\label{rk:constraint2bis}
As in Remark \ref{rk:constraint2}, we may focus on a geometrical constraint different from $\partial\Om\subset\{x, a\leq |x| \leq b\}$. In that case, {given $\mathcal{O}_{ad}$ a class of admissible open bounded set}, the previous result remain valid when replacing the definition of $(\partial\Om_{0})_{in}$ as follows
\begin{eqnarray*}
(\partial\Omega_0)_{in}
&=&
\left\{x\in\partial\Omega_0 \;/\; \exists \theta\in \S^{N-1}_{in}, x=\frac{1}{u_{0}(\theta)}\theta\right\},\; \mbox{\it with}\\[2mm]
\S^{N-1}_{in}
&=&
\{\theta\in\S^{N-1}\; /\; \exists \eps,\delta>0,  \forall v\in W^{1,\infty}_{0}(B_{\S^{N-1}}(\theta,\eps))\textrm{\it with }
\|v\|_{W^{1,\infty}}<\delta, u_0+v\in \mathcal{U}_{ad}\}\;\mbox{\it and}\\[3mm]
\mathcal{U}_{ad}
&=&
{\{u, \Om_{u}\in\mathcal{O}_{ad}\},}
\end{eqnarray*}
{if $\mathcal{U}_{ad}$ is convex, as assumed in Theorem \ref{th:multiD}}.
\end{remark}

\begin{remark}\label{rk:existmultiD}
As in Remark \ref{rk:existence}, the existence of a solution to (\ref{e:minJbisN})  can be proved easily when $0<a<b<\infty$. The proof is similar to the one in Remark \ref{rk:existence}, the only adaptation to the multi-dimensional case is the fact that
 a sequence of 1-homogeneous functions $u_{n}:\overline{B_{1}}\to\R$ (where $B_{1}$ is the unit ball of $\R^N$) such that 
 $\forall n\in\N, \forall \theta\in\S^{N-1}, \alpha\leq u_{n}(\theta)\leq \beta$, with $0<\alpha<\beta<\infty$, is strongly relatively compact in $H^1(\overline{B_{1}})$ (actually in $W^{1,p}(\overline{B_{1}})$ for any $p<\infty$). To that end, we first notice that 
 $\|\nabla u_{n}\|_{\infty}\leq \beta$ (see for example \cite{FM}), which implies that up to a subsequence $u_{n}$ converges to $u$ for the $L^\infty(\overline{B_{1}})$-norm. Clearly $u$ is then 1-homogeneous, convex, and satisfies $\forall \theta\in\S^{N-1}, \alpha\leq u(\theta)\leq \beta$. To obtain the convergence for the $H^1$-norm, we follow \cite[Lemma 2.1.2]{BB}. As in Remark \ref{rk:existence}, the perimeter, the volume, and the energy $E$ are continuous, and therefore the set whose gauge function is the limit of $u_{n}$ is a minimizer.
\end{remark}
{\bf Proof of Corollary \ref{cor:multiD}}.
Let $u_{0}$ such that $\Om_{0}=\Om_{u_{0}}$. Then $u_{0}$ is a solution of \eqref{e:minJubis}, where 
$j(u)={R}(e(u),m(u))-p(u)$ and $e(u)=E(\Om_{u}), m(u)=|\Om_{u}|, p(u)=P(\Om_{u})$.

The second order derivatives related to the geometrical terms $m(u)$ and $p(u)$ can be easily computed. Indeed, 
first note that for every $u\in W^{1,\infty}({\S^{N-1}})$, $u>0$ we have {(see for instance \cite{fug})}
\[
m(u)=\int_{\S^{N-1}} \frac{1}{Nu^N}d\theta, \;\; \qquad 
p(u)=\int_{\S^{N-1}} \frac{\sqrt{u^2+|\nabla_{\tau} u|^2}}{u^N}d\theta.
\]
where $\nabla_{\tau}$ denotes the tangential gradient on $\S^{N-1}$.
It follows that $m$ and $p$ are twice differentiable and there exists $\beta_{1}$ (depending on $u_0$) such that, 
$$
|m''(u_{0})(v,v)|\leq \beta_{1}\|v\|^2_{L^2(\T)}
$$

{
Easy direct calculations give
\begin{eqnarray*}
p'(u)(v)
&=&
\int_{\S^{N-1}}
\partial_u G(u,|\nabla_{\tau} u|)v 
+
\partial_p G(u,|\nabla_{\tau} u|)\frac{\nabla_{\tau} u\cdot \nabla_{\tau} v}{|\nabla_\tau u|},\\
p''(u)(v,v)
&=&
\int_{\S^{N-1}}
\partial_{uu}G(u,|\nabla_{\tau} u|)v^2
+
2\partial_{up}G(u,|\nabla_{\tau} u|)\frac{\nabla_{\tau} u\cdot\nabla_{\tau} v}{|\nabla_{\tau} u|}v	
+\\
& &
\hspace*{10mm}
\partial_{pp}G(u,|\nabla_{\tau} u|)\frac{(\nabla_{\tau} v\cdot \nabla_{\tau} u)^2}{|\nabla_{\tau} u|^2}
+
\partial_p G(u,|\nabla_{\tau} u|)\frac{|\nabla_{\tau} u|^2|\nabla_{\tau} v|^2-(\nabla_{\tau} v\cdot \nabla_{\tau} u)^2}{|\nabla_{\tau} u|^3}.
\end{eqnarray*}
Taking into account ${\displaystyle \partial_p G(u,p)=\frac{p}{u^N\sqrt{u^2+p^2}}}$, 
${\displaystyle \partial_{pp}G(u,p)=\frac{u^2}{u^N(u^2+p^2)^{3/2}}}$ and rearranging the last two terms gives
\begin{eqnarray*}
p''(u)(v,v)
&=&
\int_{\S^{N-1}}
\partial_{uu}G(u,|\nabla_{\tau} u|)v^2
+
2\partial_{up}G(u,|\nabla_{\tau} u|)\frac{\nabla_{\tau} u\cdot\nabla_{\tau} v}{|\nabla_{\tau} u|}v	
+\\
& &
\hspace*{10mm}
\frac{1}{u^N\sqrt{u^2+|\nabla_\tau u|^2}}
\left[
|\nabla_\tau v|^2 
-
\frac{|\nabla_\tau u|^2}{u^2+|\nabla_\tau u|^2} 
\left(\frac{\nabla_{\tau} u}{|\nabla_\tau u|}\cdot \nabla_{\tau} v\right)^2
\right].
\end{eqnarray*}
As the first and second terms are controlled by $\beta_2\|v\|^2_{L^2}$ and $\gamma\|v\|_{L^2}|v|_{H^1}$ respectively, with 
$\beta_{2}, \gamma$ depending on $u_0$, and 
\[
|\nabla_\tau v|^2 
-
\frac{|\nabla_\tau u|^2}{u^2+|\nabla_\tau u|^2} 
\left(\frac{\nabla_{\tau} u}{|\nabla_\tau u|}\cdot \nabla_{\tau} v\right)^2
\geq
(1-(1-\alpha))|\nabla_\tau v|^2 = \alpha\|\nabla_\tau v|^2,
\]
with $\alpha=\alpha(u_0)>0$, it follows
\[
p''(u_{0})(v,v)\geq \alpha|v|^2_{H^{1}(\S^{N-1})}-\gamma|v|_{H^{1}(\S^{N-1})}\|v\|_{L^2(\S^{N-1})}-\beta_{2}\|v\|_{L^2(\S^{N-1})}^2.
\]
}
{As in the proof of Theorem \ref{th:polygon}, the term $e''(u_{0})(v,v)$ relies on Section  \ref{ss:mainresults}, this time in any dimension, but restricted to the case of convex domains. 
With the same proof and using iii), Theorem \ref{th:e',e''(i,intro)},  we obtain 
\[
|e''(u_{0})(v,v)|\leq \beta_{3}(r)\|\xi''(0)\|_{W^{1-\frac{1}{r},r}(\partial\Om)}
+
\beta_{4}(r)\|\xi'(0)\|_{W^{1-\frac{1}{2r},2r}(\partial\Om)}^2,
\]
valid for any $r\in(1,\infty)$.
We would like to write the above inequality in terms of the $H^s(\mathbb S^{N-1})$-norm of $v$, for some $s\in(0,1)$.
For this, first we note that any bi-Lipschitz transformation from $\mathbb S^{N-1}$ to $\partial\Omega_0$ 
{(in our case it is $\theta\mapsto (1/u(\theta),\theta)$)}
defines a
diffeomorphism from $W^{s,p}(\partial\Omega_0)$ to $W^{s,p}(\mathbb S^{N-1})$, for every $s\in[0,1]$ and $p\in[1,\infty]$. Then the above inequality is equivalent to
\begin{equation}\label{e''(0),1}
|e''(u_{0})(v,v)|\leq \beta_{5}(r,u)\left\|{\frac{v^2}{u_0^{3}}}\right\|_{W^{1-\frac{1}{r},r}(\mathbb S^{N-1})}
+
\beta_{6}(r,u)\left\|{\frac{v}{u_0^{2}}}\right\|_{W^{1-\frac{1}{2r},2r}(\mathbb S^{N-1})}^2,
\end{equation}
Note that here we used the formulas (\ref{e:xi'(u),xi''(u)}) with $e^{i\theta}$ replaced by $\theta$ (of norm 1).

For the first term in (\ref{e''(0),1}) we 
wonder for which $s\in(0,1)$ the following inclusion is continuous:
\begin{equation}\label{e:prod1}
H^s(\R^{N-1})H^s(\R^{N-1})W^{1,\infty}(\R^{N-1})\subset W^{1-1/r,r}(\R^{N-1}).
\end{equation}
We use \cite[Corollary, p. 189]{R+S}, which refers to the more general embedding 
$B^{s_1}_{p_1,{q_1}}\cdots B^{s_m}_{p_m,{q_m}}\subset B^s_{p,{q}}$ in $\mathbb R^N$. Note that by using a partition of unity, some Lipschitz transformations and the extension theorem, we may use these embeddings as if $\mathbb S^{N-1}$ was $\mathbb R^{N-1}$, and therefore we use it with $N-1$ instead of $N$. In our specific case, this leads to the condition
\begin{equation}\label{e:s-1}
{\left(1-\frac{1}{r}\right)\frac{N}{2}<s<1}.
\end{equation}
Clearly, for every $s\in(0,1)$ one can find $r$ so that (\ref{e:s-1}) holds, which implies that \eqref{e:prod1} is valid.

For the second term in (\ref{e''(0),1}), we look for the inclusion 
\begin{equation}\label{e:prod2}
H^s(\R^{N-1})W^{1,\infty}(\R^{N-1})\subset W^{1-1/(2r),2r}(\R^{N-1}).
\end{equation}
which is valid (again using \cite{R+S}) if
\begin{equation}\label{e:s-2}
{1-\frac{1}{2r}+\frac{N-1}{2}(1-\frac{1}{r})< s<1}.
\end{equation}
Inequality (\ref{e:s-2}) has a solution $r$ (sufficiently close to $1$) for any $s\in(1/2,1)$, which implies that (\ref{e:prod2}) holds for every $s\in(1/2,1)$.

Combining (\ref{e:prod1}) and (\ref{e:prod2}) with (\ref{e''(0),1}) gives
\[|e''(u)(v,v)|\leq \beta_{7}(s,u_0)\|v\|_{H^{s}(\S^{N-1})}^2,
\]
for any $s\in(1/2,1)$.}

Easy computations about the term ${R}(E(\Om),|\Om|)$ implies  then that $j''(u_{0})$ satisfies \eqref{e:coercivitebis},
and therefore Theorem \ref{th:multiD} applies, and $T_{u_{0}}$ has finite dimension.

The fact that the Gauss curvature must vanish where it is defined is an easy consequence. Indeed, if $\omega$ is as in the statement of Corollary \ref{cor:multiD} and if the Gauss curvature is positive at some point of $\omega$, then it is positive in a neighborhood $\hat{\omega}$ of this point. As a consequence, any smooth function with compact support in $\hat{\omega}$ is in $T_{u_0}$. This contredicts the fact that $T_{u_0}$ has finite dimension.
\hfill$\Box$
}}

\bibliographystyle{plain}

\end{document}